\documentclass[12pt,reqno]{amsart}
\usepackage{geometry}       
\pdfoutput=1                   
\usepackage{graphicx}
\usepackage{amssymb}
\usepackage{epstopdf}

\usepackage{epsfig}

\usepackage{amsthm}

\def\mathscr{\mathcal}
\usepackage[all]{xy}

\theoremstyle{plain}
\newtheorem{theorem}[equation]{Theorem}
\newtheorem{proposition}[equation]{Proposition}

\newtheorem{lemma}[equation]{Lemma}

\theoremstyle{definition}
\newtheorem{definition}[equation]{Definition}
\newtheorem{remark}[equation]{Remark}
\newtheorem{example}[equation]{Example}
\numberwithin{equation}{section}

\newcommand{\Z}{{\mathbb Z}}
\newcommand{\N}{{\mathbb N}}
\newcommand{\R}{{\mathbb R}}

\renewcommand{\leq}{\leqslant}
\renewcommand{\geq}{\geqslant}
\renewcommand{\preceq}{\preccurlyeq}

\usepackage{color}
\usepackage{graphicx} %
\usepackage{epsfig}
\usepackage{amsthm}
\usepackage{epstopdf}
\def\mathscr{\mathcal}
\usepackage[all]{xy}
\usepackage{color}
\usepackage{psfrag}
\usepackage{pinlabel}
\usepackage{paralist}
\numberwithin{equation}{section}
\numberwithin{figure}{section}
\numberwithin{equation}{subsection}
\numberwithin{figure}{section}

\newtheorem{notation}[equation]{Notation}
\theoremstyle{plain}
\begin{document}

\title{Isoperimetric Inequalities using Varopoulos Transport}
\author{Antara Mukherjee}
\address{Dept.\ of Mathematics and Computer Science\\
        The Citadel,\\
        Charleston, SC 29409\\
        USA}
\email{antara.mukherjee@citadel.edu}
\date{\today}

\begin{abstract}
The main results in this paper provide upper bounds of the second order Dehn functions of three-dimensional groups Nil and Sol. These upper bounds are obtained by using the Varopoulos transport argument on dual graphs. The first step is to start with reduced handlebody diagrams of the three-dimensional balls either immersed or embedded in the universal covers of each group and then define dual graphs using the 0-handles as vertices, 1-handles as edges. The idea is to reduce the original isoperimetric problem involving volume of three-dimensional balls and areas of their boundary spheres to a problem involving Varopoulos' notion of volume and boundary of finite domains in dual graphs.
\end{abstract}

\maketitle

\section{Introduction}
\subsection{History of Filling Functions}

\paragraph{}

 The origin of the quest to find a link between topology and combinatorial group theory can be traced back to Belgian physicist Plateau's (1873, \cite{Pla}) classical question whether every rectifiable Jordan loop in every 3-dimensional Euclidean space bounds a disc of minimal area. Since then geometers and topologists have been investigating various ways to obtain efficient fillings of spheres by minimal volume balls. Thanks to the efforts of Dehn \cite{Dehn} and Gromov \cite{Gr3} we now know that there is an intimate connection between this classical geometric problem and group theory.
Various other results on Dehn functions can be found in papers by McCammond \cite{Cammond}, Ol'shanski{\u\i} \cite{Ol} and Rips \cite{Rips1}. The most significant development in this area has been Gromov's introduction of word hyperbolic groups.\\
  Important results in the area of Dehn functions using different techniques also appear in the pair of following papers, the first in 1997 (published in 2002) by Sapir, Birget and Rips (\cite{sapir}) and the second in 2002 by Birget, Ol'shanski{\u\i}, Yu, Rips and Sapir, (\cite{Ol1}). They showed that there exists a close connection between Dehn functions and complexity functions of Turing machines. One of their main results said that the Dehn function of a finitely presented group is equivalent to the time function of a two-tape Turing machine.
  More of this history and background on isoperimetric inequalities can be found in the paper by Bridson in \cite{Bri}.

Since the 1990's topologists have been interested in Dehn functions in higher dimensions. Gromov \cite{Gr2}, Epstein et al.\cite{Ep}, first introduced the higher order Dehn functions and Alonso et al. \cite{Al} and Bridson \cite{Bri1} produced the first few results in the context of these functions.

This paper would not have been possible without the support and guidance of my
advisor, Dr. Noel Brady of the University of Oklahoma.

\subsection{Goal of this research}
\paragraph{}
The main theorems of this paper provide upper bounds of the second order Dehn functions for 3-dimensional groups Nil and Sol.

\begin{theorem} [A. Mukherjee] \label{th1}
 The upper bound of the second order Dehn function (denoted by $\delta^{(2)}$) of the lattices in the Nil geometry
is given by $\delta^{(2)}(n) \preceq n^{\frac{4}{3}} $.\\
In other words, the upper bound of the second order Dehn function of the groups  $\Z^2 \rtimes_\phi \Z$, where $\phi$ has eigenvalues $\pm 1$ and has infinite order, is given by, $\delta^{(2)}(n) \preceq n^{\frac{4}{3}} $.

\end{theorem}
\begin{theorem} [A. Mukherjee]\label{th2}
 The upper bound of the second order Dehn function (denoted by $\delta^{(2)}$) of the lattices of the 3-dimensional geometry Sol is given by $\delta^{(2)}(n) \preceq n\ln (n) $.\\
 In other words, that the second order Dehn functions for the groups  $\Z^2 \rtimes_\phi \Z$, where the eigenvalues of $\phi$ are not $\pm 1$,
$\delta^{(2)}(n) \preceq n\ln (n) $.
\end{theorem}
\section{Overview of proof of the main theorems}
\paragraph{}
The main goal of this research is to obtain upper bounds of second order Dehn function of the groups mentioned in the theorems above.


In order to obtain upper bounds, we start with a reduced, transverse diagram $f:(D^3,S^2)\rightarrow K$, where $D^3$ is a 3-ball, $S^2$ is its boundary sphere and $K$ is the 3-dimensional ambient space. We then define a dual Cayley graph $\Gamma$ in the ambient space $K$ where each vertex of $\Gamma$ is a 3-cell in $K$ and each edge is a 2-cell common to two adjacent 3-cells. Now, we consider a finite subset of vertices $D$ of $\Gamma$ corresponding to the 0-handles of the diagram mapped into $K$ and we define an integer-valued function $\phi_D:\Gamma^{(0)}\rightarrow \Z^+$ with finite support i.e, $\phi_D(\alpha)=\hspace{0.1cm}$number of pre-images of $\alpha$ in $(D^3,S^2)$ for all $\alpha\in D$, otherwise $\phi_D(\alpha)=0$. This leads us to the fact that the volume of the 3-ball $D^3$ and $||\phi_D||= \displaystyle\sum_{\sigma\in D}\phi_D(\sigma)$ are equal. The boundary of $D$ according to Varopoulos is $\partial_V D= \{\tau:  \tau$  is a face of two 3-cells, $ \sigma_i, \sigma_j ;\phi_D(\sigma_i)\neq \phi_D(\sigma_j)\}$, next we define $\parallel \nabla \phi_D \parallel=
\displaystyle\sum_{\tau \in\partial_V D }| \phi_D(t(\tau))-\phi_D(i(\tau))|$, where $i,t$ are functions which determine the initial and terminal vertices of an edge in $\Gamma$. This function gives the number of edges in the boundary $\partial_V D$. In fact, we can show that $\parallel \nabla \phi_D \parallel \leq Vol^2(S^2)$. Therefore the problem of upper bound reduces to an inequality involving $||\phi_D||$ and $\parallel \nabla \phi_D \parallel$ provided $Vol^3(D^3)=||\phi_D||$ and $Vol^2(S^2)\geq ||\nabla \phi_D||$.

Finally, we show that $||\phi_D||\leq ||\nabla \phi_D||^\frac{4}{3}$ for the lattices in the 3-dimensional geometry Nil and $||\phi_D||\leq ||\nabla \phi_D||\ln(||\nabla \phi_D||)$ for the lattices in the 3-dimensional geometry Sol using a variation the Varopoulos transport argument.

\subsection{Organization of the paper}
\paragraph{}
This paper is organized as follows, in the third section we introduce ordinary Dehn functions as well as higher order Dehn functions and discuss results involving higher order Dehn functions.

In the fourth section we give a survey of generalized handle body diagrams in 2 and 3-dimensions which can be thought of as higher dimensional analogs of van Kampen diagrams. We use transverse maps for this and the main result here is to show that a reduced diagram can be obtained from an unreduced diagram without changing the map on the boundary. Reduced diagrams are a key to obtaining upper bounds for second order Dehn function.

The fifth section introduces the structure of the 3-manifolds which are torus bundles over the circle. We then describe
the cell decomposition of the torus bundles and introduce the notion of dual graphs in the cell decomposition.
Finally we focus on the main examples of this paper which are lattices in the 3-dimensional geometries Nil and Sol.

 The main result in the sixth section is that the isoperimetric inequality involving $Vol^3(D^3)$ and $Vol^2(S^2)$ reduces to an inequality between $||\phi_D||$ and $||\nabla \phi_D||$. We do this by defining a dual graph in the ambient space.

In the last section we use the Varopoulos transport argument to obtain the upper bounds of second order Dehn functions in case of both Nil and Sol.

\section{Basic Notions on Dehn Functions}\label{basic}
In this section we introduce some basic definitions on ordinary and higher dimensional Dehn functions. We also present a short survey of results involving higher dimensional later in the section. The definitions were primarily taken from \cite{Bri} and \cite{Bra}.
\subsection{Dehn Functions}\label{iso}
\begin{definition}(\emph{Dehn function}).
Let $\mathcal{P}= \langle \mathcal{A} \mid \mathcal{R}\rangle$ be the finite presentation of a group $G$, where $\mathcal{A}$ denotes the  set of generators and $\mathcal{R}$ denotes the set of all relators.\\
We can define the \emph{Dehn} \emph{function} of $\mathcal{P}$ in the following way: (\cite{Bri})\\
Given a word $w = 1$ in  generators $\mathcal{A}^{\pm 1}$, \\
$Area(w) = $ min$\{ N_w\in \N : \exists$ an equality $ w= \displaystyle\prod_{i=1}^{N_w} {x_ir_ix_i^{-1}}; x_i\in F(\mathcal{A})$ and $ r_i\in \mathcal{R} \}$, here $F(\mathcal{A})$ denotes the free group on the generating set $\mathcal{A}$.\\
The \emph{Dehn function} of $\mathcal{P}$ is $\delta_{\mathcal{P}}(n) \hspace{0.1cm} = \hspace{0.1cm}$ max$\{Area(w) : |w| \leq n \}$.
\end{definition}


\begin{definition}(\emph{Equivalent Functions}).\label{eq}
Two functions $f,g:[0,\infty)\rightarrow [0, \infty)$ are said to be $\sim$ equivalent if $f\preceq g$ and $g\preceq f$, where $f\preceq g$ means that there exists a constant $C>0 $ such that $f(x)\leq Cg(Cx)+Cx$, for all $x\geq 0$, (and modulo this equivalence relation it therefore makes sense to talk of ``the'' Dehn function of a finitely presented group). This equivalence is called \emph{coarse Lipschitz equivalence}.
\end{definition}


\begin{definition}(\emph{Isoperimetric Function of a Group}).
A function $f:\N\rightarrow \N$ is an \emph{isoperimetric function} for a group $G$ if the Dehn function $\delta_\mathcal{P} \preceq f$ for some
(and hence any) finite presentation $\mathcal{P}$ of $G$.
\end{definition}

Given a smooth, closed, Riemannian manifold $M$, in the rest of this section we shall describe the isoperimetric function of $M$ and discuss its relationship with the Dehn function of the fundamental group $\pi_1(M)$ of $M$.

Let $c:S^1\rightarrow M$ be a null-homotopic, rectifiable loop and define $FArea(c)$ to be the infimum of the areas of all Lipschitz maps
$g:D^2\rightarrow X$ such that $g|_{\partial {D^2}}$ is a reparametrization of $c$. \\
Note that the notion of area used here is the same as that
of area in spaces introduced by Alexandrov \cite{Alex}. The basic idea is to define the area of a surface (or area of a map $g:D^2\rightarrow X$) to be the limiting area of approximating polyhedral surfaces built out of Euclidean triangles.

\begin{definition} (\emph{Isoperimetric or Filling function})
Let $M$ be a smooth, complete, Riemannian manifold. The genus zero, 2-dimensional, isoperimetric function of $M$ is the function $[0,\infty)\rightarrow [0,\infty)$ defined by, $\hspace{0.5cm}{Fill}^M_0(l):= sup\{FArea(c) \hspace{0.2cm}| \hspace{0.2cm}c:S^1 \rightarrow M \hspace{0.1cm}$ null-homotopic, $ length(c)\leq l\} $.
\end{definition}
The Filling Theorem provides an equivalence between Dehn function and the Filling function defined above.
\begin{theorem}[Filling Theorem, Gromov \cite{Gr3}, Bridson, \cite{Bri}]
The genus zero, 2-dimensional isoperimetric function $Fill^M_0$ of any smooth, closed, Riemannian manifold $M$ is $\sim$ equivalent
to the Dehn function $\delta_{\pi_1 M}$ of the fundamental group of $M$.
\end{theorem}

\begin{example}
 Here are a few examples of manifolds and their Dehn functions.
\begin{enumerate}
\item The Dehn function of the fundamental group of a compact 2-manifold is linear except for the torus and the Klein bottle when it is quadratic.
 \item The groups that interest us are fundamental groups of 3-manifolds and the Dehn functions of these groups can be characterized using the following theorem by Epstein and Thurston.\\
\emph{Let $M$ be a compact 3-manifold such that it satisfies Thurston's geometrisation conjecture (\cite{Th}).\\
The Dehn function of $\pi_1(M)$ is linear, quadratic, cubic or exponential. It is linear if and only if $\pi_1(M)$  does not contain $\Z^2$.
It is quadratic if and only if $\pi_1(M)$ contains $\Z^2$ but does not contain a subgroup $\Z^2 \rtimes_\phi \Z$ with $\phi \in GL(2,\Z)$ of
infinite order. Subgroups $\Z^2 \rtimes_\phi \Z$ arise only if a finite-sheeted covering of $M$ has a connected summand that is a torus bundle
over the circle, and the Dehn function of $\pi_1(M)$ is cubic only if each such summand is a quotient of the Heisenberg group.}
\end{enumerate}
 \end{example}

\subsection{\textbf{Geometric Interpretation of the Dehn function}.}\label{geom}

\paragraph{}
The connection between maps of discs filling loops in CW complexes (or in other words a geometric interpretation of the Dehn function defined above) and the algebraic method of reducing words can be explained by any one of the following,

\begin{itemize}
  \item van Kampen diagrams (\cite{short}),
  \item pictures (\cite{short})and,
  \item Handle body diagrams (Discussed in Section 3).
\end{itemize}

\subsection{Higher dimensional Dehn functions}

\paragraph{}
Epstein et al. \cite{Ep} and Gromov \cite{Gr2} first introduced higher dimensional Dehn functions at about the same time.
 However later, Alonso \emph{et al.} \cite{Al} and  Bridson \cite{Bri1} provided equivalent definitions which were different from the two mentioned above. In the discussion on higher dimensional Dehn functions presented here we will be using  Brady\emph{ et al}'s (\cite{Bra}) definition which is based on the prior definitions given by Bridson and Alonso\emph{ et al}. Before we introduce higher dimensional Dehn functions we note the definition of groups of type $\mathcal{F}_n$.

 \begin{definition}(\emph{Eilenberg-MacLane complex}, \cite{Bri2})
 The Eilenberg-MacLane complex (or classifying space) $K(\Gamma,1)$ for a group $\Gamma$ is a CW complex with fundamental group $\Gamma$ and contractible universal cover. Such a complex always exists and its homotopy type depends only on $\Gamma$.
 \end{definition}

 \begin{definition}(\emph{Finiteness property $\mathcal{F}_n$}, \cite{wall})
 A group $\Gamma$ is said to be of type $\mathcal{F}_n$ if it has an Eilenberg-MacLane complex $K(\Gamma,1)$ with finite $n$-skeleton.
 Clearly a group is of type $\mathcal{F}_1$ if and only if it is finitely generated and of type $\mathcal{F}_2$ if and only if it is finitely presented.
 \end{definition}

Intuitively, the $k$-\emph{dimensional Dehn function}, $k\geq 1$ ,
is the function $\delta^{(k)}: \N \rightarrow \N$ defined for any group $G$ which is of type $\mathcal{F}_{k+1}$ and $\delta^{(k)}(n)$ measures
the number of $(k+1)$-cells that is needed to fill any singular $k$-sphere in the classifying space $K(G,1)$, comprised of at most $n \hspace{0.2cm} k$-cells.  Up to equivalence the higher dimensional Dehn functions of groups are quasi-isometry invariants.

The following part of this section is devoted to the technical definition of higher dimensional Dehn function given by Brady\emph{ et al.}, (\cite{Bra}).

\begin{notation}
Henceforth we will denote an $n$-dimensional disc (or ball) by $D^{n}$ and an $n$-dimensional sphere by $S^{n}$.
\end{notation}

\begin{definition} (\emph{Admissible maps})\label{adm}
Let $W$ be a compact $k$-dimensional manifold and $X$ a CW complex,
an admissible map is a continuous map $f: W \rightarrow X^{(k)} \subset X$ such that $f^{-1}(X^{(k)}- X^{(k-1)})$ is a
disjoint union of open $k$-dimensional balls, each mapped by $f$ homeomorphically onto a $k$-cell of $X$.
\end{definition}

\begin{definition} \label{volume}
(\emph{Volume of $f$}) If  $f: W \rightarrow X$ is admissible we define the volume of $f$, denoted by $Vol^{k}(f)$, to be
the number of open $k$-balls in $W$ mapping to $k$-cells of $X$.
\end{definition}

Given a group $G$ of type $\mathcal{F}_{k+1}$, fix an aspherical CW complex $X$ with fundamental group $G$
and finite $(k+1)$-skeleton. Let $\widetilde{X}$ be the universal cover of $X$. If $f: S^k \rightarrow \widetilde{X}$ is an
\emph{admissible} map, define the \emph{filling volume} of $f$ to be the minimal volume of an extension of $f$ to $D^{k+1}$ in the following way,
FVol$(f) \hspace{0.2cm} = \hspace{0.2cm}$min$\{Vol^{k+1}(g)\hspace{0.1cm}|\hspace{0.1cm}g: D^{k+1}\rightarrow \widetilde{X}, g|_{\partial D^{k+1}} =f \}$, then, $k-$dimensional Dehn function of $X$ is $\delta^{(k)}(n) \hspace{0.2cm} = \hspace{0.2cm} $sup$\{$ FVol$(f)\hspace{0.1cm} |\hspace{0.1cm}f: S^{k}\rightarrow \widetilde{X}, Vol^{k}(f)\leq n \}$.

\begin{remark} Here are a few observations about higher dimensional Dehn functions,
\begin{enumerate}
  \item Up to equivalence, $\delta^{(k)}(n)$ is a quasi-isometry invariant.
  \item In the above definitions it is possible to use $X$ in place of $\widetilde{X}$ since $f: S^k \rightarrow {X}$ (or $f: D^{k+1} \rightarrow {X}$) and their lifts to $\widetilde{X}$ have the same volume.
\end{enumerate}

\end{remark}
All the groups discussed in this paper is at most 3-dimensional so we will restrict $k$ in the above definitions such that $k\leq 2$.



The following are examples of second order Dehn functions.
\begin{example}
(\emph{Examples of groups and their second-order Dehn functions}):

\begin{enumerate}
  \item By definition,the second order Dehn function of a 2-complex with contractible universal cover is linear.

  \item The second order Dehn function of any group of every (word) hyperbolic group $H$ is linear and so is the direct product of $H$ with any finitely generated free group, both these results were established by Alonso \emph{et al.} in \cite{Bog}.


  \item The second order Dehn function of any finitely generated abelian group with torsion-free rank greater that two is $\sim n^{3/2}$,
            e.g, $\Z^3 $ (\cite{Wang2}).

\end{enumerate}
\end{example}

\section{Transverse Maps, Handle Decompositions and Reduced Diagrams}\label{maps}
In this section we will discuss generalized handle decompositions which will help us compute upper bounds of higher dimensional Dehn functions in specific cases later in the paper.

\newpage
\subsection{Background on Handle Decompositions}

\paragraph{}
Any compact, smooth or piecewise linear manifold, admits a handle decomposition (\cite{Mil}, \cite{Rk}), also each handle decomposition can be made proper (see details in \cite{Rk}). In 1961 S. Smale \cite{Sm}, established the existence of exact handle decompositions of simply connected and cobordisms of dimensionality $n\geq 6$.

In this paper we will be using the generalized handle decomposition of manifolds, mainly due to
Buoncristiano, Rourke, Sanderson, \cite{Bu}.
 This reference by Buoncristiano, Rourke, Sanderson (\cite{Bu}) is a lecture series on a geometric approach to homology theory.

Here they introduce the concept of transverse CW complexes. These complexes have all the same properties of ordinary cell complexes. The result from this article which we will be using in this paper is known as the Transversality Theorem, and using this theorem any continuous map may be homotoped to a transverse map (Definition \ref{transverse}). Here is the statement of the Transversality theorem, this theorem is used to show the maps from the handle decompositions we construct to the ambient space are transverse.

\begin{theorem} [Buoncristiano, Rourke and Sanderson, \cite{Bu}]
Suppose $X$ is a transverse CW complex (a CW complex is transverse if each attaching map is transverse to the skeleton to which it is mapped), and
$f:M \rightarrow X$ is a map where $M$ is a compact piecewise linear manifold. Suppose $f\mid_{\partial M}$ is transverse, then there is a homotopy of $f\hspace{0.1cm} rel \hspace{0.1cm} \partial M$ to a transverse map.
 \end{theorem}
In fact, if $M$ is a generalized handle decomposition i.e, it is constructed from another manifold with boundary $M_0$, by attaching finite number of generalized  handles, then the map $f$ itself is homotopic to a transverse map.

\subsection{Handlebody Diagrams}\label{transverse1}

\paragraph{}
The following definitions and statement of Transversality theorem were taken from the lecture notes of a course \cite{For} taught by Max Forester at the University of Oklahoma.

\begin{definition} \label{handle}
(\emph{Index i-Handle}) An index $i$-handle is written as $H^i=\Sigma^i \times D^{n-i}$, where $\Sigma^i$ is a connected $i$-manifold (we will consider $\Sigma^i = D^i$ in all our examples) and $D^{n-i}$ is a $(n-i)$ closed disk.
\end{definition}

\textbf{Note}: The boundary of a $i$-handle is $\partial H^i = \partial \Sigma^i \times D^{n-i} \cup \Sigma^i \times \partial D^{n-i}$.\\

Given an $n$-manifold $M_0$ with boundary and an $i$-handle $H^i$, let $\phi: \partial \Sigma^i \times D^{n-i} \rightarrow \partial M_0$ be an
embedding. Form $M_0 \cup_\phi H^i$ a new manifold with boundary obtained from $M_0$ by attaching an $i$-handle in the following way,
 $(M_0 \amalg H^i)/(x \sim \phi(x), \forall x \in \partial \Sigma^i \times D^{n-i})$.

\begin{definition} \label{filtratiion}
(\emph{Generalized Handle Decomposition}) A generalized handle decomposition of $M$ is a filtration:
$\emptyset = M^{(-1)} \subset M^{(0)} \subset M^{(1)} \subset ......\subset M^{(n)} = M$ such that:
\begin{itemize}
\item Each $M^{(i)}$ is a codimension-zero submanifold of $M$.
($L \subset M$ is a codimension-zero submanifold if $L$ is an $n$-manifold with boundary and $\partial L$ is a submanifold of $M$.)
\item $M^{(i)}$ is obtained from $M^{(i-1)}$ by attaching finitely many $i$-handles.
\end{itemize}
\end{definition}
\begin{remark}
In case $M$ is a compact $n$-manifold with boundary denoted by, $\partial M$, then the generalized handle decomposition of $M$ is:
\begin{itemize}
\item A generalized handle decomposition of $\partial M$, namely: $\emptyset =N^{(-1)} \subset N^{(0)} \subset N^{(1)} \subset ......\subset N^{(n-1)} = \partial M$, where each $M^{(i)}$ is a codimension-zero submanifold of $\partial M$
     \item A filtration of $M$, $\emptyset = M^{(-1)} \subset M^{(0)} \subset M^{(1)} \subset ......\subset M^{(n)} = M$ where each $M^{(i)}$ is a codimension-zero submanifold of $M$ and $M^{(i)}$ is obtained from $M^{(i-1)}\cup N^{(i-1)}$ by attaching $i$-handles.
    \item Each $(i-1)$-handle of $N$ is a connected component of the intersection of $N$ with an $i$-handle of $M$ (this means that $N^{(i-1)}=\partial M\cap M^{(i)}$).
\end{itemize}
\end{remark}
\begin{definition} \label{transverse}
(\emph{Transverse Maps}) Let $M$ be a compact $n$-manifold and $X$ a cell-complex. A continuous map $f:M \rightarrow X $ is \emph{transverse} if $M$ has a generalized handle decomposition such that for every handle $H^i=\Sigma^i \times D^{n-i}$ in $M$, the restriction $f\mid _{H^i} : \Sigma^i \times D^{n-i} \rightarrow X$ is given by $\phi \circ {pr}_2$ where ${pr}_2: \Sigma^i \times D^{n-i} \rightarrow D^{n-i}$ is a projection map
to the second coordinate and $\phi$ is the characteristic map of an $(n-i)$-cell of $X$. We will refer to the generalized handle decomposition of $M$ as a \emph{handle body diagram} or just a \emph{diagram}.
\end{definition}

\textbf{Note} An $i$-handle maps to a $(n-i)$-cell this implies, $f(M) \subset X^{(n)} =X$.

\begin{definition} (\emph{``good''  CW complex})\label{good}
A CW complex is ``good'' if and only if, each attaching map is transverse to the skeleton to which it is mapped.
\end{definition}
Next we have a version of the Transversality theorem which we will refer to later in this paper.

\begin{theorem}[Transversality Theorem \cite{Bu}]\label{Ttheorem}
If $X$ is an $n$-dimensional, ``good''  CW complex and $M$ is a generalized handle decomposition of a compact $n$-manifold, then every continuous map $f: M \rightarrow X$ is homotopic to a transverse map $g$. Moreover, if $f \mid _{\partial M}$ is transverse, then there is a homotopy  of $f$ rel $\partial M$ to a transverse map.
\end{theorem}

\begin{lemma}\label{complex}
Every cell-complex is homotopy equivalent to a ``good'' cell-complex.
 \end{lemma}

 \begin{definition}(\emph{Unreduced Diagram}). \label{reduced}
A  diagram  $f: (D^n,S^{n-1}) \rightarrow K$ is said to be \emph{unreduced } if in the interior of $(D^n,S^{n-1})$ there exists two 0-handles $H^0_1$ and $H^0_2$ joined together by a 1-handle such that, $f(H^0_1)=f(H^0_2)$ is an open $n$-cell in $K$ and $({f|_{H^0_1}}^{-1}\circ f|_{H^0_2})$ is an orientation reversing map. Otherwise, the diagram is said to be \emph{reduced}.
\end{definition}

 In other words, a diagram is \emph{unreduced} if there exists another diagram with the same boundary length or area (in case of 2 or 3-dimensional cases respectively) but strictly smaller filling area or volume for 2 or 3-dimensional cases respectively.
   Under these circumstances we will eliminate these 0-handles along with the 1-handle connecting them but keeping the boundary of the diagram same and ensuring that we still have a disc. Hence, our intention is to get a \emph{reduced} diagram from an \emph{unreduced} one.

   Next we will discuss how to obtain a reduced diagram from an unreduced one. This argument was given by Brady and Forester (\cite{For1}).\\

\begin{example}\label{annul}
Let $f: (D^n, S^{n-1}) \rightarrow K$ be an admissible map, and let $H^0_1$ and $H^0_2$ be 0-handles in $(D^n, S^{n-1})$ connected together with a 1-handle. Let $\alpha$ be a core curve in the 1-handle connecting $H^0_1$ and $H^0_2$ homeomorphic to an interval (Figure \ref{alpha}). Suppose $f$ maps $\alpha$ to a point and maps $H^0_1$ and $H^0_2$ to the same $n$-cell, with opposite orientations. As $H^0_1$ and $H^0_2$ are 0-handles, there are homeomorphisms
$h_i:(H^0_i, \partial H^0_i)\rightarrow (D^n,S^{n-1}) $ such that $f|_{H^0_i}= \phi \circ h_i$ for some characteristic map $\phi :(D^n,S^{n-1}) \rightarrow K$. We first consider the curve $\alpha$ along with  a tubular neighborhood around it and collapse it to a point to get part $(ii)$ of Figure \ref{alpha}.
Next remove the interiors
of $H^0_i$ from $(D^n, S^{n-1})$ and form a quotient $(D^n_1, S^{n-1}_1)$ by gluing boundaries via $h_0^{-1}\circ h_1$, an orientation reversing map. The new space maps to $K$ by $f$, and there is a homeomorphism $g:(D^n, S^{n-1}) \rightarrow (D^n_1, S^{n-1}_1)$. Now $f\circ g$ is an admissible map $(D^n, S^{n-1}) \rightarrow K$ with two fewer 0-handles. The map can be then be made transverse with the rest of the 0-handles unchanged. Figure \ref{alpha} illustrates the method pictorially.

\begin{figure}[ht!]
\labellist
\small\hair 2pt
\pinlabel $H_2^0$  at 170 239
\pinlabel $H_1^0$  at 314 239
\pinlabel $\alpha$  at 262 230
\pinlabel $(i)$  at 240 181
\pinlabel $(ii)$  at 136 -6
\pinlabel $(iii)$  at 451 -6
\pinlabel $f$  at 141 192
\pinlabel $g$  at 293 91
\pinlabel $H_2^0$  at 80 44
\pinlabel $H_1^0$  at 162 46
\endlabellist
\centering
\includegraphics[scale=0.65]{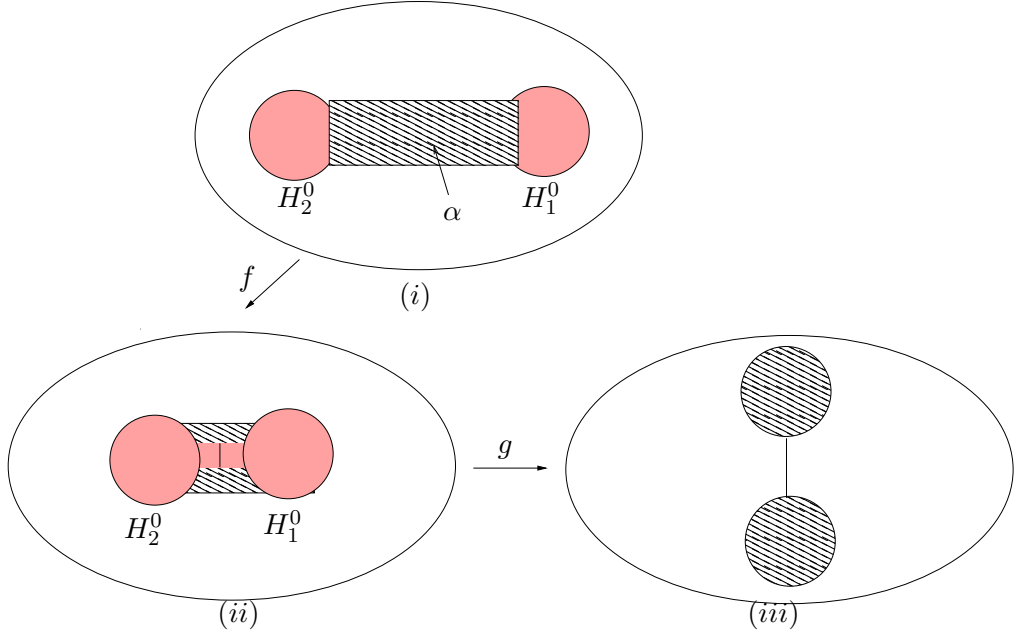}
\caption{$(i)$ two 0-handles joined by a 1-handle and core curve $\alpha$, $(ii)$ Picture of $(i)$ after $\alpha$ has been removed, and $(iii)$ Final Picture}
\label{alpha}
\end{figure}

 \end{example}

  \section{The connection between Linear Algebra and Cell Decomposition of
Mapping Tori} \label{decomposition}
In this section we discuss the structure of the 3-dimensional manifolds that have the lattices of the Nil and Sol geometries as fundamental  groups.

The 3-manifolds considered here are the mapping tori where the attaching maps corresponds to matrices
in $SL_2(\Z)$. In other words, given a  group of the form $\Z^2 \rtimes_{\psi_A} \Z$,
 where $\psi_A \in Aut(\Z^2)$ and $A \in SL_2(\Z)$, the geometric realization of these groups are mapping tori where the attaching maps are the automorphisms of $\Z^2$. For
example if $\psi_A$ is the identity map then, the corresponding space is
$\Z^3 \subset \R^3$. Other specific examples we are interested in are the lattices in the 3-dimensional geometries Nil and Sol. In
particular we will be looking at lattices corresponding to the matrix $\left(%
\begin{array}{cc}
  1 & 1 \\
  0 & 1 \\
\end{array}%
\right)$ for Nil and $\left(%
\begin{array}{cc}
  2 & 1 \\
  1 & 1 \\
\end{array}%
\right)$ in case of Sol.
Another way of looking at these are as torus bundles over the circle and they are described below.

Let us denote the mapping torus $\frac{T \times I}{(t,0)\sim
(\psi_A(t),1)} $ by $A_T$, where $\psi_A$ is the attaching map.
Let $\psi_A$ be represented by the matrix $A\equiv$ $\left(%
\begin{array}{cc}
  x & z\\
  y & w \\
\end{array}%
\right)\in SL_2(\Z)$ . So, if the generating curves of the torus
in $A_T$ are labeled $a,b$, then the presentation of the corresponding fundamental group is given by,
$ \Gamma=\langle a,b, t \mid [a,b],
tat^{-1}=A(a)=a^xb^y,\hspace{0.15cm}tbt^{-1} = A(b)=a^zb^w \rangle$.\\

\subsection{Cell Decomposition of the Mapping Torus $A_T$}

\paragraph{}
We know that the mapping torus $A_T$ consists of two copies of the
torus attached via the map $\psi_A$.
 Here we will demonstrate an effective way of
triangulating the 2-cell spanned by the generators of the group $\Gamma$ and hence obtain a model space for $\Gamma$.

We subdivide the 2-cells of both copies of the torus in $A_T$ into either a number of triangular faces or a combination of triangular and quadrilateral faces.
The following example illustrates this process in details.
\begin{example}\label{lattice1}
Let $A \equiv \left(%
\begin{array}{cc}
  1 & 1 \\
  0 & 1 \\
\end{array}%
\right)$, then the corresponding group is the 3-dimensional, integral Heisenberg
group $\mathcal{H} = \langle a,b,t \hspace{0.15cm}\mid
\hspace{0.15cm} [a,b],\hspace{0.1cm} tat^{-1}=a,\hspace{0.1cm}
tbt^{-1}=ab \rangle$.

\begin{figure}[ht!]
\labellist
\small\hair 2pt
\pinlabel $a$  at 152 192
\pinlabel $b$  at 196 154
\pinlabel $a$  at 152 110
\pinlabel $b$  at 110 153
\pinlabel $a$  at 396 110
\pinlabel $b$  at 500 157
\pinlabel $a$  at 458 190
\pinlabel $b$  at 356 153
\pinlabel $\psi_A$  at 274 133
\endlabellist
\centering
\includegraphics[scale=0.65]{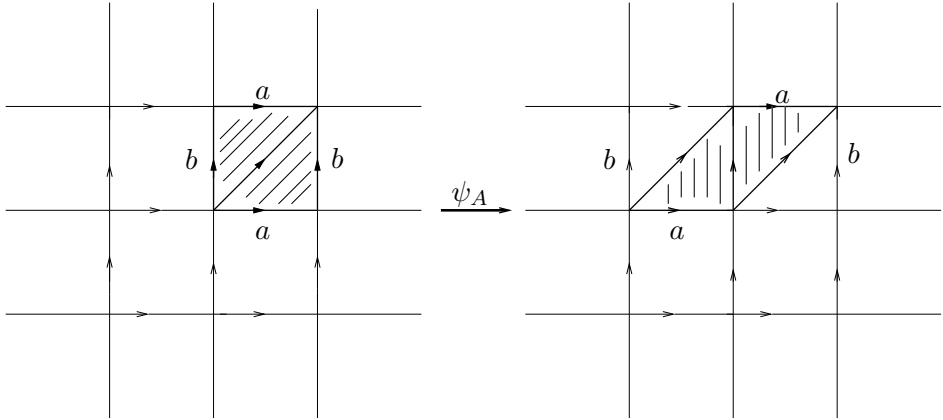}

\caption{Sub-division of $\R^2$ under the action of the map $\psi_A$}
\label{grid}
\end{figure}



\begin{figure}[ht!]
\labellist
\small\hair 2pt
\pinlabel $A$ at -5 -10
\pinlabel $B$ at -4 110
\pinlabel $C$ at 55 160
\pinlabel $D$ at 146 -5
\pinlabel $E$ at 206 51
\pinlabel $F$ at 206 160
\pinlabel $G$ at 145 105
\pinlabel $H$ at 58 58
\pinlabel $a$ at 81 24
\pinlabel $b$ at 122 66
\pinlabel $c$ at 28 31
\pinlabel $t$ at 215 105
\pinlabel $b$ at 65 -5
\pinlabel $b$ at 125 162
\pinlabel $a$ at 24 136
\pinlabel $a$ at 172 119
\pinlabel $c$ at 83 134
\pinlabel $c$ at 164 20
\pinlabel $b$ at 82 98
\pinlabel $t$ at -5 58

\endlabellist
\centering
\includegraphics[scale=0.65]{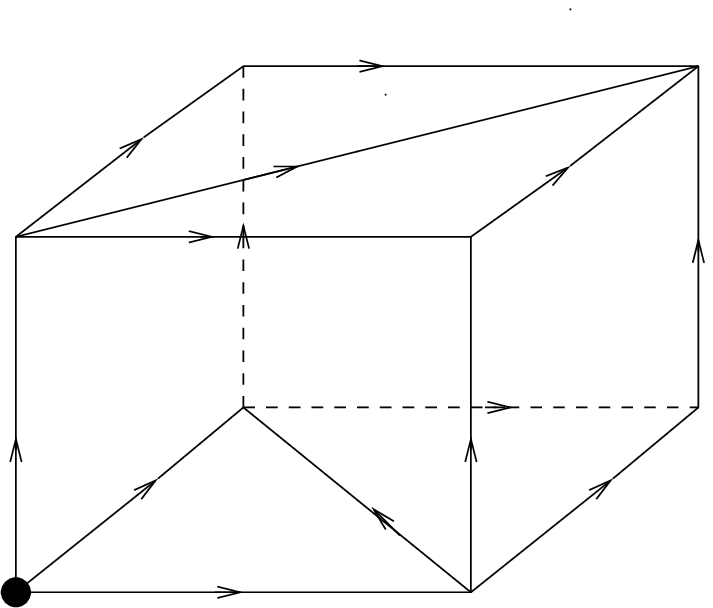}

\caption{The mapping torus corresponding to the matrix $A$ above.} \label{nil}
\end{figure}

This subdivision  of the mapping tori below, (Figure \ref{nil})  shows that
 the top has been divided into two triangular faces each of which can be mapped via $\psi_A$ to their
 exact replicas in base. The 3-cell in (Figure \ref{nil}) also serves as the fundamental domain for the action of $\mathcal{H}$ on the corresponding universal cover. The base point is named $A$ and all other vertices of the cell are also labeled.
\end{example}

\begin{example}\label{lattice2}
If we have the matrix $B \equiv \left(%
\begin{array}{cc}
  2 & 1 \\
  1 & 1 \\
\end{array}%
\right)$, then the corresponding group presentation is $\mathcal{S} =
\langle a,b,t \hspace{0.15cm}\mid \hspace{0.15cm}
[a,b],\hspace{0.1cm} tat^{-1}=a^2b,\hspace{0.1cm} tbt^{-1}=ab
\rangle$. \\

\begin{figure}[ht!]
\labellist
\small\hair 2pt
\pinlabel $b$ at 120 81
\pinlabel $b$ at 45 81
\pinlabel $a$ at 75 43
\pinlabel $a$ at 81 114
\pinlabel $b$ at 386 73
\pinlabel $b$ at 383 133
\pinlabel $a$ at 401 99
\pinlabel $a$ at 347 101
\pinlabel $a$ at 458 101
\pinlabel $\psi_B$ at 240 117
\endlabellist
\centering
\includegraphics[scale=0.65]{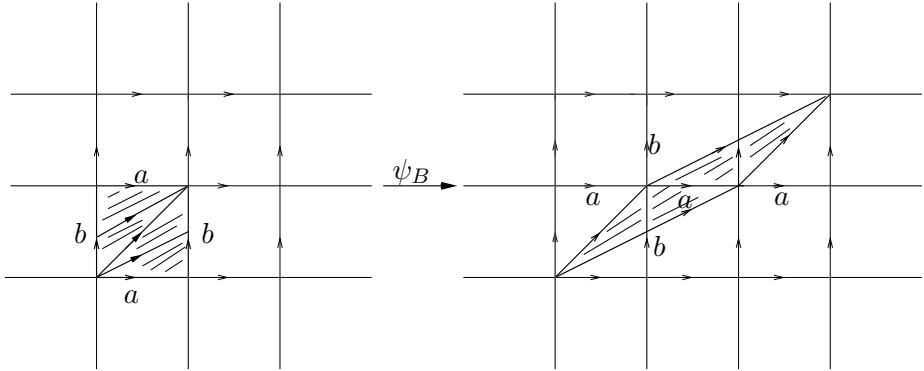}

\caption{Sub-division of two copies of $\R^2$ under the action of  $\psi_B$  .} \label{solgrid1}
\end{figure}

 Again (Figure \ref{solgrid1}) above shows that the subdivision is
 compatible to the relations in the group presentation
 $\mathcal{S}$.

\begin{figure}[ht!]
\labellist
\small\hair 2pt
\pinlabel $A$ at 4 -4
\pinlabel $B$ at -2 100
\pinlabel $C$ at 31 157
\pinlabel $D$ at 166 -2
\pinlabel $E$ at 195 51
\pinlabel $F$ at 197 157
\pinlabel $G$ at 173 97
\pinlabel $H$ at 26 51
 \pinlabel $b_1$ at 141 34
 \pinlabel $b_2$ at 44 24
 \pinlabel $b_1$ at 2 115
 \pinlabel $b_2$ at 19 146
 \pinlabel $b_1$ at 180 110
 \pinlabel $b_2$ at 190 140
 \pinlabel $c_1$ at 148 116
 \pinlabel $c_2$ at 63 141
 \pinlabel $c_1$ at 66 58
 \pinlabel $c_2$ at 144 58
 \pinlabel $c_1$ at 34 -2
 \pinlabel $c_2$ at 120 -2
 \pinlabel $a$ at 95 19
 \pinlabel $a$ at 99 155
 \pinlabel $t$ at 172 71
 \pinlabel $t$ at 200 91
 \pinlabel $t$ at 0 69
 \pinlabel $M_1$ at 185 125
 \pinlabel $M_2$ at 5 131
 \pinlabel $M_3$ at 100 57
 \pinlabel $M_4$ at 81 -2
 \pinlabel $d$ at 55 126
\pinlabel $d$ at 14 31
\pinlabel $d$ at 189 26
%
 \endlabellist
\centering
\includegraphics[scale=0.85]{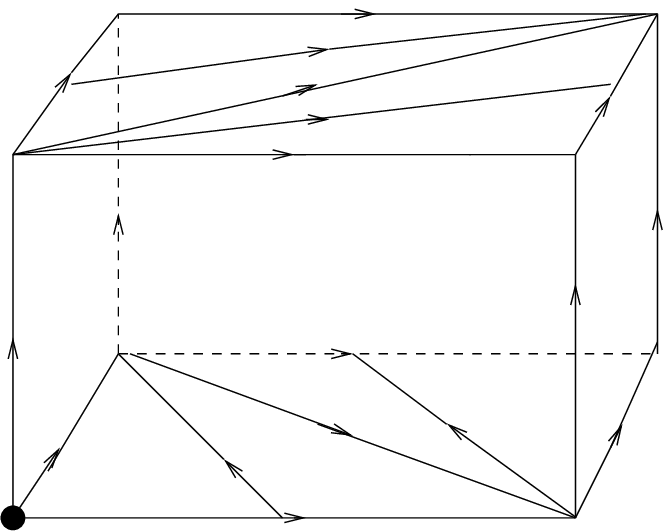}

\caption{The mapping torus corresponding to the matrix $B$ above.} \label{sol}
\end{figure}

 This triangulation  of the mapping tori below, (Figure \ref{sol})  shows that
 the top has been divided into four triangular faces each of which can be mapped via $\psi_B$ to their
exact replicas in base.
 This 3-cell serves as the fundamental domain for the action of $\mathcal{S}$ on the corresponding universal cover.
\end{example}

\section{Upper Bounds- Reduction to Varopoulos Isoperimetric Inequality}
   Sections 6 and 7 are devoted to obtaining upper bounds for the second order Dehn functions of $\mathcal{H}$ and $\mathcal{S}$ using a variation of Varopoulos Transport argument.

   In section 6 we reduce the original isoperimetric problem involving volume of 3-balls and areas of their boundary 2-spheres to a problem involving
   Varopoulos' notion of volume and boundary of finite domains in dual graphs.\\

\subsection{Definitions}
Since we will use barycentric subdivisions to obtain the dual graph, we will start this section with the following definitions.
(These definitions and notations have been taken from\cite{Bri}.)
\begin{definition} (\emph{Barycentric Subdivision of a convex
polyhedral cell}) Let $C$ be a polyhedral cell in an
$n$-dimensional polyhedral complex $K$. The barycentric
subdivision of $C$ denoted by $C'$
is the simplicial complex defined as follows:\\
There is one geodesic simplex in $C'$ corresponding to each
strictly ascending sequence of faces $F_0 \subset F_1
\subset....\subset F_n$ of $C$; the simplex is the convex hull of
barycenters of $F_i$. Note that the intersection in $C$ of two
such simplices is again such a simplex. The natural map from the
disjoint union of these geodesic simplices to $C$ imposes on $C$
the structure of a simplicial complex - this is $C'$.
\end{definition}

\begin{definition} (\emph{Barycentric Subdivision of a polyhedral
$n$-complex $K$}) Let $p: \coprod_\lambda C_\lambda \rightarrow K$
(where $C_\lambda$ are the polyhedral cells of $K$), be a
projection. For each cell $C_\lambda$ we index the simplices of
the barycentric subdivision $C'_{\lambda}$ by a set $I_\lambda$;
so $C'_{\lambda}$ is the simplicial complex associated to
$\coprod_{I_\lambda}S_i \rightarrow C_\lambda$ where $S_i$ denotes
the simplices of $C'_\lambda$. Let $\lambda'= \coprod_\lambda
I_\lambda$. By composing the natural maps $\coprod_{I_\lambda}S_i
\rightarrow C_\lambda $ and $p: \coprod_\lambda C_\lambda
\rightarrow K$ we get a projection $p': \coprod_{i\in \lambda'}S_i
\rightarrow K$. Let $K'$ be the quotient of $\coprod_{i\in
\lambda'}S_i$ by the equivalence relation $[x\sim y$ iff $p'(x)=
p'(y)]$. $K'$ is the barycentric subdivision of $K$.
\end{definition}

\textbf{Note} Given any complex, there is a poset $\mathcal{P}$ on
the cells of the complex ordered by inclusion. Therefore for any
ascending chain in $\mathcal{P}$ there is a simplex in the
barycentric subdivision of the complex.

\subsection{Dual Graphs}\label{dual}

\paragraph{}
The examples in the previous section gives us an idea of the cell
decomposition of the spaces under consideration.
The groups considered here are all finitely generated,
so the groups act properly and cocompactly by isometries on their respective universal covers. In fact,
the translates of the fundamental domain covers the universal cover $\widetilde{X}$ in each case.

It is essential to mention here that the only groups we are interested in are the 3-dimensional groups $\mathcal{H}$ and $\mathcal{S}$ from Section \ref{decomposition} and we will use the letter $G$ to refer to them in general.

Next, we define the dual graph $\Gamma$ using Definition \ref{graph}. The vertex set of $\Gamma$, $V_\Gamma=\{ \sigma: \sigma$ is
 a 3-cell of $\widetilde{X} \}$ while the edge set is, $E_\Gamma =\{ \tau: \tau$ is a codimension one face (2-cells) shared by two adjacent 3-cells of $K \}$.

\begin{lemma}
There is a map that embeds the graph $\Gamma$ in $K$.
\end{lemma}
\begin{proof}
Consider the barycentric subdivision of both the graph $\Gamma$ and the universal cover $\widetilde{X}$, we denote these
barycentric subdivisions by $\Gamma'$ and $\widetilde{X'}$ respectively. Next we map the vertices in $V_\Gamma$ to the barycenters of the 3-cells
while we map the barycenter of an edge $\tau$, labeled by $\tau_m$ in $E_\Gamma$ to the barycenter of the codimension one face shared by the two 3-cells in $V_\Gamma$, serving as the initial and terminal vertices of $\tau$. Finally, if $\tau$ is an edge with initial and terminal vertices $\sigma_1$ and $\sigma_2$ respectively, then, the left half-edge of $\tau$ is mapped to
the simplex in $K'$ corresponding to the ascending chain $\tau \subset \sigma_1$ in the poset $\mathcal{P}$ while, the right half-edge maps to the simplex in $K'$ corresponding to the ascending chain $\tau\subset \sigma_2$ in $\mathcal{P}$.

 As there is a natural bijection between the barycentric subdivision of a space and the geometric realization of the space itself
so, there is a map that embeds $\Gamma$ in $K$.
\end{proof}

So, now we have a dual graph in $\widetilde{X}$ which is also a Cayley graph (with the same name $\Gamma$), with respect to a finite generating set which we will define subsequently. The aim of the remaining part of this section is to show that $\Gamma$ is quasi-isometric to $\widetilde{X}$ using the following lemma.

\begin{lemma}($\check{S}varc-Milnor$ Lemma,\cite{Bri})
Given a length space $X$. If a group $G$ acts properly and cocompactly by isometries on $X$, then $G$ is finitely generated and for any choice of
basepoint $x_0 \in X$, the map $f: G \rightarrow X$, defined by $g \mapsto g.x_0$ is a quasi-isometry.
\end{lemma}

  Let $C$ be the fundamental domain of $\widetilde{X}$ ( a compact subset of $\widetilde{X}$ such that its translates covers all of $\widetilde{X}$). We then define the generating set of the group $G$ in the following way, $\mathcal{A}= \{ g\in G \hspace{0.2cm}\mid \hspace{0.2cm} gC \cap C =codimension-one \hspace{0.1cm} face \}$.

  In case of $\mathcal{H}$ the valence of a
 vertex  is eight, while in the case of $\mathcal{S}$ the valence is twelve. Hence, the generating sets in these cases will contain
four and six elements respectively. We will define the generating sets in detail for specific examples i.e, for the groups $\mathcal{H}$ and $\mathcal{S}$ in the following lemma.\\

\textbf{Note}:
In the following lemma, we shall denote the triangular faces of the cell decomposition obtained in the previous
 section as $\triangle XYZ$, where $X,Y,Z$ are the labels of vertices in the cell decomposition forming a triangle.

\begin{lemma} \label{generator} Given the cell decompositions for groups $\mathcal{H}$ and $\mathcal{S}$ in section \ref{decomposition}:
\begin{enumerate}
\item $\mathcal{A}_0 = \{ b,c,t, tb\} $ is a finite generating set for  $\mathcal{H}$, where $c=b^{-1}a$ (from Figure \ref{nil}).
\item $\mathcal{A}_0 = \{ d,t,c_1c_2,td^{-1},tc_2^{-1}c_1^{-1},tb_1c_1^{-1}\} $ is a finite generating set for  $\mathcal{S}$, where,
$a_1a_2=a, d=ba=ab, c_1= ab_1, c_2=b_2a $ (from Figure \ref{sol}).
\end{enumerate}
\end{lemma}

\begin{proof} (\emph{of (1)})
 We consider Figure \ref{nil} for this part of the proof. The vertex $A$ is chosen as the base point  of universal cover $\widetilde{X}$.
 The paths that take the base point to its images in copies
of the fundamental domain (which are 3-cells sharing codimension one faces with the fundamental domain) represent the isometries
that take the domain to its copies and hence they are the generators of the group with respect to the Cayley graph $\Gamma$. In case of $\mathcal{H}$, there are eight other 3-cells sharing codimension one faces with the fundamental domain or in other words, due to the cell decomposition
 shown in section \ref{decomposition}, any 3-cell in the universal cover shares a codimension one face with eight other 3-cells.

 In the following lines we give a list of isometries and hence the words which generate translates of the fundamental domain that share a codimension one face with the domain.

The path from $A$ to $D$ represents the isometry $b$ taking the domain to the 3-cell to its right; path from $A$ to $H$ represents the word $c$ takes the domain to the cell behind itself; $A$ to $B$, the word $t$ takes the domain to the 3-cell on the face $\triangle BCG$; path from $A$ to $G$, the word $tb$ takes the domain to the 3-cell on the face $\triangle GFC$. The isometries that take the domain to the rest of the neighboring 3-cells, are inverses of the words already mentioned above. For example the isometry taking the domain to the 3-cell sharing the face $\triangle ADE$ is $t^{-1}$, while the one taking it to the 3-cell associated with the face $\triangle AHE$ is $b^{-1}t^{-1}$ etc. So it is clear that $\mathcal{A}_0 = \{ b,c,t, tb\} $  is a finite generating set for $\mathcal{H}$ and ${\mathcal{A}_0}^{-1} = \{b^{-1},c^{-1},t^{-1},b^{-1}t^{-1} \}  $.

(\emph{Proof. of (2)}) This can be shown in a similar way as above. In this case, the fundamental domain shares codimension one faces
with twelve other 3-cells, (four cells each above and below, two on each side and the remaining two at the front and back).
 As before the translates $d$ and $t$ generate copies to the right and
vertically above (and sharing the face $\triangle BFM_1$)  the fundamental domain respectively. The translate $td^{-1}$ generates the copy sharing the face $\triangle BGM_1$, while $tc_2^{-1}c_1^{-1}$ generates the copy of the fundamental domain along the face $\triangle BM_2F$. Finally $tb_1c_1^{-1}$ is responsible for the copy of the domain sharing the face $\triangle M_2CF$ with the fundamental domain. So $\mathcal{A}_0 = \{ d,t,c_1c_2,td^{-1},tc_2^{-1}c_1^{-1},tb_1c_1^{-1}\} $. Also, it is easy to check that ${\mathcal{A}_0}^{-1} = \{d^{-1}, t^{-1}, dt^{-1},c_1c_2t^{-1},c_1b_1^{-1}t^{-1}\}$.
\end{proof}

\begin{proposition} \label{cay}
Cay($G, \mathcal{A}_0$), the Cayley graph of the group $G$ with respect to the generating sets $\mathcal{A}_0$ defined in Lemma \ref{generator} is quasi-isometric to $\widetilde{X}$ .
\end{proposition}
\begin{proof}
Milnor's Lemma says that the group $G$ is finitely generated and quasi-isometric to the ambient space $\widetilde{X}$. But the
Cayley graph Cay$(G,\mathcal{A})$ with respect to any finite generating set  $\mathcal{A}$ of the group $G$, is quasi-isometric to the group itself, this quasi-isometry can be seen as the natural inclusion $G \hookrightarrow Cay(G,\mathcal{A})$, defined by $g \mapsto g.1$ for all $g \in G$. This last quasi-isometry is also a simple illustration of Milnor's Lemma.

Finally, two Cayley graphs associated to the same group but with different generating sets are quasi-isometric, this implies Cay$(G, \mathcal{A}_0)$ is quasi-isometric to $\widetilde{X}$.
\end{proof}

\subsection{Definitions and Notations} \label{var}

\paragraph{}
\vspace{0.3cm}
 We start with the definition of a dual graph (Section \ref{dual}).

\begin{definition} \label{graph}
Given an ambient $n$-dimensional space $K$,
 we define a graph $\Gamma$ with
vertex set  $V_\Gamma
\hspace{0.2cm}=\hspace{0.2cm}\{\sigma: \hspace{0.2cm}\sigma$ is a
$n$-cell of $K\}$ and edge set  $E_\Gamma\hspace{0.2cm}=\hspace{0.2cm}\{\tau: \hspace{0.2cm}\tau$ is a
$(n-1)$-cell and $\tau$ is a face of exactly 2 $n$-cells
of $K\}$.
\end{definition}

Given a finitely presented group $G$, let $X$ be the corresponding $n$-dimensional cell-complex and let $\widetilde{X}$ be its universal cover. Let  $f: (D^n, S^{n-1})\rightarrow \widetilde{X}$ be a reduced diagram (defined in Section \ref{transverse1}) where $D^n$ and its boundary sphere $S^{n-1}$ are either embedded or immersed in $\widetilde{X}$. Note that the map $f$ considered here is transverse and hence admissible, so each $i$-handle in the diagram maps to an $(n-i)$-cell in $\widetilde{X}$. Next, we consider a finite subset $D$ of the vertex $V_\Gamma$ such that,
 $D=\{\sigma: \sigma $ is an $n$-cell in $\widetilde{X}$ such that $\sigma \in Im(f)\}$. Associated with $D$ is a function analogous to a characteristic map, given by, $\phi_D: V_\Gamma \rightarrow \mathbb{N}\cup \{0\}$ defined
by, $\phi_D(\sigma)=$ number of pre-images of $\sigma$ under $f$.

\begin{remark}\label{phi1}
Let $||\phi_D||= \displaystyle\sum_{\sigma\in D}\phi_D(\sigma)$, this is the number of 0-handles in the diagram i.e, $||\phi_D||= Vol^n(D^n)$ where $Vol^n(D^n)$ denotes the volume of the $n$-ball $D^n$.
\end{remark}
\begin{remark}
It is clear that if $f$ is an embedding in the above definition then $\phi_D$ is in fact the characteristic function of the set $D$.
\end{remark}

\begin{definition} \label{vboundary} The Varopoulos boundary of $D$ is
defined to be the set of all $(n-1)$-cells $\tau \in E_\Gamma$
such that $\tau$ is a face of exactly two $n$-cells
$\sigma_i,\sigma_j \in V_\Gamma$ such that $\phi_D(\sigma_i)\neq \phi_D(\sigma_j)$.\\

\textbf{Notation}: The Varopoulos boundary will be denoted by, $\partial_V D$.
\end{definition}

Next we define $\nabla \phi_D:E_\Gamma \rightarrow \N \cup \{0\}$
  by, $\nabla \phi_D(\tau)= |\phi_D(t(\tau))-\phi_D(i(\tau))
  |$, where $i$ and $t$ have the same definition as before.\\
The cardinality of the Varopoulos boundary $|\partial_V D|$, in this
case can be given by, \\
$\hspace{2.5cm}\parallel \nabla \phi_D \parallel=
\displaystyle\sum_{\tau \in\partial_V D }| \phi_D(t(\tau))-\phi_D(i(\tau))
|$. Note that this definition says that $\tau \in E_\Gamma$ is a boundary edge of $D$
if $\phi_D(t(\tau)) \neq \phi_D(i(\tau))$.

\subsection{Reducing to Varopoulos Isoperimetric Inequality }

\paragraph{}
In this section we show that our problem to obtain an upper bound for the second order Dehn functions can be reduced to finding an inequality between volume and boundary notions according to Varopoulos in case of $\mathcal{H}$ and $\mathcal{S}$. We start with the following lemma which works in general for dimensions 1 or more.

\begin{lemma}  \label{bound}
$\parallel \nabla \phi_D \parallel \leq | \partial
D^n |$, where $|\partial D^n |$ is the
area or volume of the boundary sphere of the diagram $(D^n,S^{n-1})$ for $n>1$.
\end{lemma}
\begin{proof}
Let us consider the n-dimensional reduced diagram $g: (D^n,S^{n-1}) \rightarrow \widetilde{X}$ (Definition \ref{reduced}).
Let $\tau \in \widetilde{X}$ be the $(n-1)$-cell such that $i(\tau)= \sigma_1$ and $t(\tau)= \sigma_2$,  for $\sigma_1,\sigma_2 \in D$.
In terms of poset $\mathcal{P}$, $\tau \subset \sigma_1$
and $\tau \subset \sigma_2$ where $\sigma_1,\sigma_2$ are $n$-cells
in $\widetilde{X}$ such that $\sigma_1,\sigma_2 \in D (\subset
V_\Gamma)$ and $\phi_D(\sigma_1)\neq \phi_D(\sigma_2)$.

\begin{figure}[ht!]
\labellist
\small\hair 2pt

\pinlabel {$\sigma_1^k$} at 85 116
\pinlabel {$\tau^k$} at 133 116
\pinlabel {$\sigma_2^l$} at 178 116
\pinlabel {$\tau^j$} at 314 107
\pinlabel {$\sigma_1^i$} at 130 53
\pinlabel {$\tau^i$} at 135 13
\pinlabel {$\sigma_2^j$} at 275 108
\pinlabel {$\tau^j$} at 314 107

 \endlabellist
\centering
\includegraphics[scale=0.85]{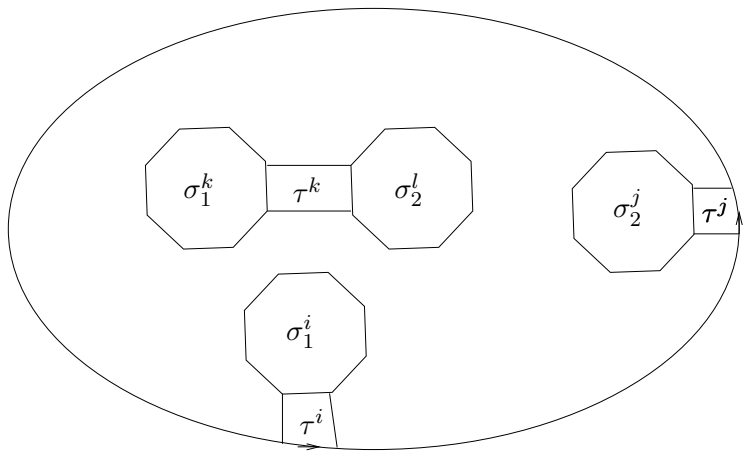}

\caption{2-dimensional example with pre-images $\sigma_1$ , $\sigma_2$ and 1-cell $\tau$ in $(D^2,S^1)$} \label{cells}
\end{figure}

By the definition of $\phi_D$, there are $\phi_D(\sigma_1)$
0-handles in $(D^n,S^{n-1})$ that map onto $\sigma_1$ via $g$ and
similarly there are $\phi_D(\sigma_2)$ 0-handles in $(D^n,S^{n-1})$ that
map onto $\sigma_2$ via $g$. Next, since we have
$\phi_D(\sigma_1)\neq \phi_D(\sigma_2)$, this implies
$\tau$ is one of the $(n-1)$-cells forming the boundary $(n-1)$-sphere, i.e, $\tau \in \partial_V D$ and as both $n$-cells have more than one pre-images, thus, $\tau$ too has one or more pre-images in $(D^n,S^{n-1})$ associated with
pre-images of both $\sigma_1$ and $\sigma_2$. The pre-images of $\sigma_1$ and $\sigma_2$ are either in the interior
of $(D^n,S^{n-1})$ with pre-images of $\tau$ or they are at the boundary with $\tau$ as a boundary $(n-1)$-cell in some instances.

\begin{figure}[ht!]
\labellist
\small\hair 2pt

\pinlabel {$\sigma_1^i$} at 50 140
\pinlabel {$\tau^i$} at 116 116
\pinlabel {$\sigma_2^i$} at 181 148
\pinlabel {$\sigma_1^j$} at 71 58
\pinlabel {$\tau^k$} at 150 41

\endlabellist
\centering
\includegraphics[scale=0.85]{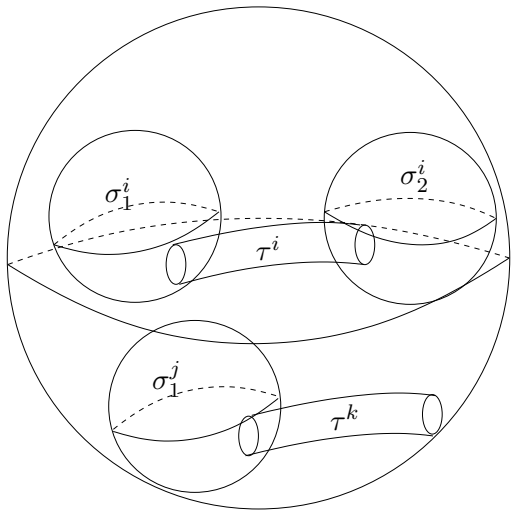}
\caption{3-dimensional example with pre-images for  cells $\sigma_1$ , $\sigma_2$ and 2-cell $\tau$ in $(D^3,S^2)$ } \label{cells1}
\end{figure}

If all the pre-images of $\sigma_1$ and $\sigma_2$ are in the interior of $(D^n,S^{n-1})$ with all pre-images of $\tau$ in the interior, then this implies $\phi_D(\sigma_1) = \phi_D(\sigma_2)$, which is against our assumption.
Without loss of generality let us assume that $\phi_D(\sigma_1) > \phi_D(\sigma_2)$.
In this case if at most $\phi_D(\sigma_2)$ of the pre-images are in the interior of $(D^n,S^{n-1})$,
then as we are considering handle decomposition of $n$-balls which are manifolds, the only way a pre-image
of $\sigma_2$ appears in the interior is if it is accompanied with a pre-image of $\sigma_1$ and they share a pre-image
of $\tau$ which is a 1-handle. Figures \ref{cells} and \ref{cells1} are illustrations of this in two and three dimensions respectively, where $\tau^i$ denotes a pre-image of $\tau$ while $\sigma_k^i$ etc.
denotes the pre-images of $\sigma_k$ for $k=1,2$. In these figures, one pre-image of $\tau$, a 1-handle, is in the interior of $(D^n,S^{n-1})$ between pre-images of $\sigma_1$ and $\sigma_2$, while the other is at the boundary adjoined to the 0-handle which is another pre-image of $\sigma_1$ . This implies that at least $(\phi_D(\sigma_1) - \phi_D(\sigma_2)) = (\phi_D(i(\tau) - \phi_D(t(\tau))$ of the pre-images of $\sigma_1$ are at the boundary of $(D^n,S^{n-1})$ with $\tau$ as a boundary $(n-1)$-cell.
Thus, $|\partial D^n | \geq \displaystyle \sum_{\tau\in \partial_V D} | \phi_D(i(\tau) - \phi_D(t(\tau)|  $
which implies, $\parallel \nabla \phi_D \parallel \leq |\partial D^n|$.
\end{proof}

\textbf{Note}: At this point, the problem involving the volume of the balls $Vol^n(D^n)$ and the area or volume of the boundary sphere $|\partial D^{n-1}|$, has reduced to one involving $||\phi_D||$ and $||\nabla\phi_D||$. In the next section we are going to use Varopoulos transport argument to prove the isoperimetric inequality involving $||\phi_D||$ and $||\nabla\phi_D||$. In case of the group $\mathcal{H}$ we will show that $||\phi_D||\leq const.||\nabla\phi_D||^{\frac{4}{3}} $ and in the case of $\mathcal{S}$, we will show that $||\phi_D||\preceq const.||\nabla\phi_D||\ln(||\nabla\phi_D||)$. These inequalities automatically provide upper bounds for the second order Dehn functions in both cases.

\section{Upper Bounds- Varopoulos Transport Argument}\label{isop}

\paragraph{}
In this section we are going to use Varopoulos transport to obtain isoperimetric inequalities in case of groups $\mathcal{H}$ and $\mathcal{S}$.
We are going to consider reduced diagrams, since in case they are unreduced
we can always use Proposition \ref{reduced} from Section \ref{transverse1} to obtain a reduced diagram.
 As before, we will denote the volume of an $n$-ball by $|D^n|$ and the volume of its boundary by $|\partial D^n|$, for any dimension $n$.

The Varopoulos isoperimetric inequality and Dehn functions have very little in common with each other. The only cases where they appear likely to agree are when the groups are fundamental groups of manifolds and also we are considering only top dimensional Dehn functions. So, in the cases we have here we can apply Varopoulos transport to obtain the isoperimetric inequality and hence the upper bounds of second order Dehn functions.

\subsection{Intuition behind the Varopoulos argument}\label{intui}

\paragraph{}
In this section, we present the intuition behind the notion of
transportation of mass from a finite-volume subset of a space. It is important to note here that all our examples are finitely presented
groups and the space under consideration will be the universal covers associated to the groups.

The following argument is originally due to
Varopoulos \cite{V}. It was used by Gromov in \cite{Gr} to demonstrate the transportation of mass (volume) in $\R^n$ and also that of a finite subset of group. This notion of transport was first described by Varopoulos in
\cite{V}, where he described transport in association with random walks. The same argument was further discussed by Gromov in \cite{Gr}.
Gromov also used this argument in his paper on Carnot-Carath\'{e}odory spaces \cite{G}. The lemma here is appropriately called ``Measure Moving lemma'' and helps in the proof of isoperimetric inequalities of hypersurfaces in Carnot-Carath\'{e}odory manifolds.
Before going into the technical details of the argument in Section
\ref{comp}, we will sketch the idea behind the argument
and the reason it works, in this section.

Given a graph $\Gamma$ let $D$
be a finite subset of the vertices of the graph transported by a path
$\gamma$, then the amount of mass transported through the
boundary of $D$ is obviously bounded above by $(|\gamma |
vol(\partial D))$. But we have to find a particular $\gamma$ to bound
$vol(D)$ by $vol(\partial D)$, for this we compute average
transport. Transport of $D$ corresponding to some $\gamma$
is defined as the mass of $D$ that is moved out of $D$ by the
action of $\gamma$. In other words it is the number of vertices in
the set $(D\gamma \setminus D)$,
where $D\gamma \hspace{0.1cm}=\hspace{0.1cm} \{v\gamma \mid v \in D\} $.

\begin{figure}[ht!]
\labellist
\small\hair 2pt

\pinlabel {$\gamma$} at 80 160
\pinlabel {$D$} at 48 74
\pinlabel {$D\gamma$} at 152 81
\endlabellist
\centering
\includegraphics[scale=0.75]{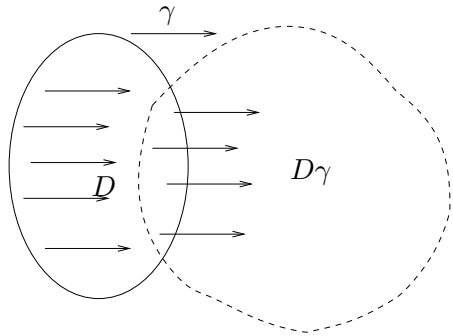}
\caption{Transport of $D$ } \label{D}
\end{figure}
Next for the lower bound for the transport we have to show that it is possible to move a percentage of the set $D$ off it.
It is always possible to choose the path $\gamma$ such that $l_\gamma$ is large enough that almost all of $D$ is transported off itself, but the key is to find a $\gamma$ in the graph such that it is small enough and moves at least half of $D$ off itself.
Since the shape of
$D$ maybe very unpredictable Figure \ref{avg}, therefore transport
via a path $\alpha$ maybe very small compared to the
mass of $D$ again for another path $\gamma$ the transport
maybe very large. In order to solve this problem we bound the length of the
path by considering a
ball of radius $R$ in $\Gamma$, denoted by $B(R)$
such that $|B(R)| \approx 2|D|$ and taking the
average transport over all $\gamma \in B(R)$. Once we
show that the average transport is at least half of $D$, we know
that there is at least one path $\gamma_0$ such that the
transport of $D$ via $\gamma_0$ is at least half of the mass of
$D$. This inequality in turn leads to the respective isoperimetric
inequalities of the groups we discuss in this context.

\begin{figure}[ht!]
\labellist
\small\hair 2pt

\pinlabel {$\gamma$} at 61 23
\pinlabel {$D$} at 48 74
\pinlabel {$\alpha$} at 109 160
\endlabellist
\centering
\includegraphics[scale=0.75]{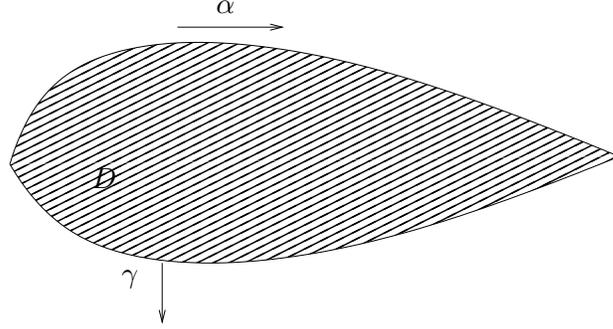}
\caption{Transport of $D$ with $\alpha$ is small compared to the
mass of $D$ while that with respect to $\gamma$ is large compared
to $D$ } \label{avg}
\end{figure}


\subsection{The Transport Computation}\label{comp}

\paragraph{}
Given a finitely presented group $G$, let the $\Gamma$ be the dual Cayley graph (defined in Section \ref{var}) corresponding to the universal cover of $n$-complex $X$ corresponding to $G$.
 This graph is infinite but it is locally finite. The edges are directed and
labeled, also there is only one outgoing (incoming) edge with a
given label at any vertex. $\Gamma$ is a Cayley graph with respect to the presentation of the groups defined in Section \ref{dual}. Also the graph is endowed with
the path metric and each edge is isomorphic to the unit interval $[0,1]$.

As defined in the previous section, in the following discussion
the vertex set of $\Gamma$ will be denoted by $V_\Gamma$ and edge
set by $E_\Gamma$. Next, consider the subset $D$ in $V_\Gamma$ corresponding to the $n$-cells in the image of $f:(D^n,S^{n-1})\rightarrow X$. Let us consider the case when $f$ is an embedding.
Then we denote the map $\phi_D$ by the characteristic function $\chi_D : V_\Gamma \rightarrow
 \{0,1\}$ defined by $\chi_D(\sigma)= 1$ when $\sigma \in D$, otherwise $\chi_D(\sigma)=0$. In this case $||\chi_D||=|D|$, where $|D|$ denotes the number of vertices in $D$.

Next, $\nabla\chi_D : \partial_V D \rightarrow \{0,1\}$ is defined in the following way,\\
 $\nabla\chi_D(\tau) = |\chi_D(t(\tau)) - \chi_D(i(\tau))|$, where $i,t: E_\Gamma \rightarrow V_\Gamma$ gives the initial and
 terminal vertices respectively of any edge in $E_\Gamma$.

Therefore, $| \partial_{V} D |\hspace{0.1cm}=\hspace{0.1cm}
\displaystyle\sum_{\tau \in \partial_{V} D} \hspace{0.2cm}| \chi_D(t(\tau)) - \chi_D(i(\tau))
|$.

Let $\gamma \in {B}(r)\subset \Gamma$, where ${B}(r)$ represents a ball of radius $r$ in the graph.
 We choose $r$ large enough such that $|{B}(r)| \geq 2| D |> |{B}(r-1)|$.\\

Varopoulos Transport $T_D^\gamma
\hspace{0.2cm}=\hspace{0.2cm}|D\gamma \setminus D|$\\

Average Transport
$\widehat{T_D^\gamma}\hspace{0.2cm}=\hspace{0.2cm}\frac{1}{|{B}(r)|}\displaystyle\sum_{\gamma\in
{B}(r)} \hspace{0.2cm} T_D^\gamma$.\\

The following is a variation of an argument given by Varopoulos, \cite{V}.

\begin{proposition}\label{ave}$\hspace{0.5cm}\widehat{T_D^\gamma}\hspace{0.2cm}\geq
 \hspace{0.2cm}\frac{1}{2}|D|$.
 \end{proposition}
 \begin{proof}

$\hspace{0.3cm}\widehat{T_D^\gamma}\hspace{0.2cm}=\hspace{0.2cm}\frac{1}{|{B}(r)|}\displaystyle\sum_{\sigma,\gamma}|
\{(\sigma,\gamma)\hspace{0.2cm}|
\hspace{0.2cm}\sigma \in D , \hspace{0.2cm}\sigma\gamma \in (V_\Gamma
\setminus D),\hspace{0.2cm}\gamma \in B(r)\}\hspace{0.2cm}|$\\

$\hspace{2.0cm}= \hspace{0.2cm}
\frac{1}{|{B_1}(r)|}\displaystyle\sum_{\gamma\in B(r)}\displaystyle\sum_{\sigma \in
D}\hspace{0.2cm}\left(\chi_D(\sigma) - \chi_D(\sigma\gamma)\right)$\\

$\hspace{2.0cm}=\hspace{0.2cm}\displaystyle\sum_{\sigma \in
D}\hspace{0.2cm}\frac{1}{| B(r)|}\displaystyle\sum_{\gamma \in
B(r)}\hspace{0.2cm}\left(\chi_D(\sigma)-\chi_D(\sigma\gamma)\right)$\\

$\hspace{2.0cm}=\hspace{0.2cm}\displaystyle\sum_{\sigma\in
D}\hspace{0.2cm}\left(\hspace{0.2cm}\frac{| B(r)|}{|
B(r)|}\hspace{0.2cm}\chi_D(\sigma)\hspace{0.2cm} -\hspace{0.2cm}
\frac{\displaystyle\sum_\gamma \hspace{0.2cm} \chi_D(\sigma\gamma)}{|
B(r)|}\hspace{0.2cm}\right)$\\

$\hspace{2.0cm}=\hspace{0.2cm}\displaystyle\sum_{\sigma\in D}\hspace{0.2cm} \left(
1 \hspace{0.2cm}- \hspace{0.2cm}\frac{|
B_\sigma(r)\cap D |}{| B(r)|}\right)$, where $B_\sigma(r)$ is a ball  of radius $r$ at vertex $\sigma$.\\

But since we assumed that $|{B}(r)| \hspace{0.2cm}>
2|D|$, so we have,\\

$\hspace{2.65cm}\widehat{T_D^\gamma}\hspace{0.2cm}\geq
\hspace{0.2cm}\displaystyle\sum_{\sigma \in
D}( 1 - \frac{1}{2})$,\\

 $\hspace{1.6cm}$ or,
$\hspace{0.2cm}\widehat{T_D^\gamma}\hspace{0.2cm}\geq
\hspace{0.2cm}\frac{1}{2}|D|$\\

So there is ${\gamma_0} \in B_1(r)$ such that $T_D^{\gamma_0} \geq
\frac{|D|}{2} $.

\end{proof}


Next we obtain an upper bound for the transport $T_D^\gamma$ in the following proposition.
\begin{proposition}
\label{lemma 1} $T_D^\gamma\hspace{0.2cm} \leq
\hspace{0.2cm}l_\gamma \hspace{0.2cm}| \partial_{V} D|$; where
$l_\gamma$ is the length of $\gamma$.
\end{proposition}
\begin{proof}
The path corresponding to the word $\gamma$ can be expressed as a
sequence of the generators in $\Gamma$, namely,
$a_1a_2a_3....a_{l_\gamma}$
where $a_i= \alpha^{\pm1}$ or $=\beta^{\pm1}$ for $1\leq i \leq l_\gamma$ .\\

\textbf{\textsl{Notation}:} Let $a_1a_2...a_k = \alpha_k$ for $1\leq k \leq
l_\gamma$ and $\alpha_0 $ is the identity of the group.\\

The  Varopoulos Transport  as defined before is,

$\hspace{1.0cm}T_D^\gamma \hspace{0.2cm}=\hspace{0.2cm}|D\gamma \setminus D|$

$\hspace{0.5cm} \therefore
T_D^\gamma\hspace{0.2cm}=\hspace{0.2cm}\displaystyle\sum_{\sigma\in D}|
\{(\sigma,\gamma)\hspace{0.2cm}|
\hspace{0.2cm}y \in D , \hspace{0.2cm}\sigma\gamma \in (V_\Gamma
\setminus D) ,\hspace{0.2cm}\gamma \in
B_1(r)\}\hspace{0.2cm}|$

$\hspace{1.65cm}= \hspace{0.2cm} \displaystyle\sum_{\sigma \in
D}\hspace{0.2cm}| \chi_D(\sigma) - \chi_D(\sigma\gamma)|$

Now using the sequence and notation defined above, we can write,

$\hspace{1.0cm}T_D^\gamma
\hspace{0.2cm}\leq\hspace{0.2cm}\displaystyle\sum_{\sigma \in
D}\hspace{0.2cm}\left(\displaystyle\sum_{i=1}^{l_\gamma}|\chi_D(\sigma\alpha_i)
- \chi_D(\sigma \alpha_{i-1})|\right)$\\

So in the inner sum, in the expression above, the terms have value
either $0$ or $1$, the terms which have value $1$, represent
boundary edges.

In order to establish the upper bound for the transport of $D$ by
$\gamma \in \Gamma$, we will show that each of the boundary edge
mentioned above appears at most $l_\gamma$ times in the sum. So,
we start with the transport of a vertex $\sigma_i \in D$ via the
path $\gamma$. Let us denote the edge between the vertices $\sigma_i\alpha_{j-1}$ and $\sigma_i\alpha_j$
by $\tau$ where $1\leq j \leq l_\gamma$.
Now let us express the path $\gamma\hspace{0.2cm}$ as the
sequence $\gamma_1\tau\gamma_2$; where
$\gamma_1,\gamma_2$ are two sub-paths of $\gamma$ such that the initial vertex of $\gamma_1$ is $\sigma_i$ while the
terminal vertex of $\gamma_2$ is $\sigma_i\gamma$, and
$\tau$ is the label of the $j^{th}$ edge of $\gamma$. Then, by uniqueness of path liftings in a Cayley
graph, it is known that $\gamma_1$ and $\gamma_2$ are both unique
with respect to initial vertex $\sigma_i$. In other words, the
paths corresponding to $\gamma$ originating from vertices of $D$
other than $\sigma_i$, do not have $\tau$ as the $j^{th}$ edge. So if the
path $\gamma$ originating from vertex $\sigma_j \in D$, (where
$\sigma_j\neq \sigma_i$), can be expressed as $\gamma_3
\tau\gamma_4$, then, here $\tau$ is the label
for say the $k^{th}, (k\neq j)$ edge of this path while $\gamma_3$
and $\gamma_4$ are both unique sub-paths with respect to
$\sigma_j$. So a particular edge in path $\gamma$ can appear at
most $l_\gamma$ times.

Therefore, $T_D^\gamma \hspace{0.2cm}\leq \hspace{0.2cm} l_\gamma
\displaystyle\sum_{\sigma_i,\sigma_j \in V_\Gamma}\hspace{0.2cm}|
\chi_D(\sigma_i)-\chi_D(\sigma_j)|$.\\

 So, $T_D^\gamma \hspace{0.2cm}\leq \hspace{0.2cm} l_\gamma
\hspace{0.2cm} |\partial_{V} D | $.
\end{proof}

Next we will consider $f:(D^n,S^{n-1})\rightarrow X$ to be an immersion, and so instead of a characteristic function we consider a
  non-negative, integer-valued function $\phi_D$ (Section \ref{var}) and show that the Varopoulos argument works in this case too.
Assume as before, that $\gamma \in {B_1}(r)\subset G$, where
${B_1}(r)$ represents a ball of radius $r$ centered at the
identity in the graph. We choose $r$ large enough such that $|{B_1}(r)| \geq 2\parallel \phi_D \parallel > |{B_1}(r-1)|$.

Varopoulos Transport $T_D^\gamma = \displaystyle\sum_{\sigma \in D} |
\phi_D(\sigma)-\phi_D(\sigma \gamma) | $.

$\therefore$ Average Varopoulos Transport is given by,

$\hspace{2.5cm}\widehat{T_D^\gamma}\hspace{0.2cm}=\hspace{0.2cm}\frac{1}{|{B_1}(r)|}\displaystyle\sum_{\gamma\in
{B_1}(r)} \hspace{0.2cm} T_D^\gamma$ .\\

\textbf{Note}: The definitions of $||\phi_D||$, $||\nabla\phi_D||$ used below can be found as Remark \ref{phi1} and Definition \ref{vboundary} respectively in Section \ref{var}.\\

The following result is a variation of an argument given by Coulhon and Saloff-Coste, \cite{Co}.\\

\begin{proposition}  \label{prop1}
$\widehat{T_D^\gamma} \geq \frac{1}{2}||\phi_D||$
\end{proposition}
  \begin{proof}
$\hspace{3.0cm}
\widehat{T_D^\gamma}=\hspace{0.2cm}\frac{1}{|{B_1}(r)|}\displaystyle\sum_{\gamma\in
{B_1}(r)}\displaystyle\sum_{\sigma \in D} | \phi_D(\sigma)-\phi_D(\sigma \gamma) | $\\

$\hspace{4.5cm}=\hspace{0.2cm}\displaystyle\sum_{\sigma \in
D}\frac{1}{|{B_1}(r)|}\displaystyle\sum_{\gamma\in
{B_1}(r)} | \phi_D(\sigma)-\phi_D(\sigma \gamma) | $\\

$\hspace{4.5cm}\geq \hspace{0.2cm}\displaystyle\sum_{\sigma \in
D}\frac{1}{|{B_1}(r)|}\displaystyle\sum_{\gamma\in {B_1}(r)} |
\phi_D(\sigma)|- | \phi_D(\sigma \gamma) |$\\

$\hspace{4.5cm}\geq \hspace{0.2cm}\displaystyle\sum_{\sigma \in
D}\frac{1}{|{B_1}(r)|}\hspace{0.2cm} \left(|
B_1(r)| \phi_D(\sigma)- \displaystyle\sum_{\gamma\in
{B_1}(r)}\phi_D(\sigma \gamma)\right)$\\

$\hspace{4.5cm}\geq \hspace{0.2cm}\displaystyle\sum_{\sigma \in
D}\hspace{0.2cm}\left( \phi_D(\sigma) -
\frac{1}{|{B_1}(r)|}\displaystyle\sum_{\gamma\in
{B_1}(r)}\phi_D(\sigma\gamma)\right) $\\

Since, $\displaystyle\sum_{\gamma\in {B_1}(r)}\phi_D(\sigma\gamma)
\hspace{0.2cm}
\leq \hspace{0.2cm}\parallel \phi_D \parallel $, we have,\\

 $\hspace{3.5cm} \widehat{T_D^\gamma}\hspace{0.2cm}\geq \hspace{0.2cm}\displaystyle\sum_{\sigma \in
D}\hspace{0.2cm}\left(  \phi_D(\sigma) - \frac{\parallel
\phi_D \parallel}{| B_1(r) |}\right)$\\

According to our initial assumption, $\frac{\parallel \phi_D
\parallel}{| B_1(r) |} \leq \frac{1}{2}$, and that implies,
$\frac{\parallel \phi_D \parallel}{| B_1(r) |} \leq
\frac{\phi_D(\sigma)}{2}$, for any particular $\sigma \in D$.\\

$\hspace{3.5cm}\therefore \widehat{T_D^\gamma}\hspace{0.2cm}\geq
\hspace{0.2cm}\displaystyle\sum_{\sigma \in D} \frac{\phi_D(\sigma)}{2} =
\frac{1}{2}\parallel \phi_D \parallel$\\

In particular, $\exists \gamma_0 \in B_1(r)$ such that,
$T_D^{\gamma_0}
\geq \frac{\parallel \phi_D \parallel}{2}$.
\end{proof}

\begin{proposition}\label{prop2}
 $T_D^\gamma \leq l_\gamma||\nabla\phi_D||$, where $l_\gamma$ denotes the length of the
path/word $\gamma$.
\end{proposition}
\begin{proof}
We will use the same argument as in proof of Lemma \ref{lemma 1} to
show this. The path corresponding to the word $\gamma$ can be
expressed as before by a sequence of the generators in $\Gamma$,
namely, $a_1a_2a_3....a_{l_\gamma}$
where $a_i= \alpha^{\pm1}$ or $=\beta^{\pm1}$ for $1\leq i \leq l_\gamma$ .\\

\textbf{\textsl{Notation}:} Let $a_1a_2...a_k = \alpha_k$ for $1\leq k \leq
l_\gamma$ and $\alpha_0 $ is the identity of the group.\\

The Varopoulos Transport as defined before is,

$\hspace{2.0cm}T_D^\gamma \hspace{0.2cm}=\hspace{0.2cm}|D\gamma \setminus D|$

$\hspace{1.7cm} \therefore
T_D^\gamma\hspace{0.2cm}=\hspace{0.2cm}\displaystyle\sum_{\sigma\in D}\hspace{0.2cm}| \phi_D(\sigma) - \phi_D(\sigma\gamma)|$

As before,

$\hspace{2.0cm}T_D^\gamma
\hspace{0.2cm}\leq\hspace{0.2cm}\displaystyle\sum_{\sigma \in
D}\hspace{0.2cm}\left(\displaystyle\sum_{i=1}^{l_\gamma}|\phi_D(\sigma\alpha_i)
- \phi_D(\sigma \alpha_{i-1})|\right)$\\

The terms in the inner sum are either zero or a natural number. In
the case when they are non-zero, they represent boundary edges in
the Varopoulos sense.

So as in the proof of Lemma \ref{lemma 1}, each of these
afore-mentioned boundary edges appear in the sum at most
$l_\gamma$ times. Therefore,

$\hspace{4.0cm}T_D^\gamma \hspace{0.2cm} \leq
\hspace{0.2cm}l_\gamma \displaystyle\sum_{\sigma_i\sigma_j \in V_\Gamma}|
\phi_D(\sigma_i)-\phi_D(\sigma_j)| $,

which means, $T_D^\gamma \hspace{0.2cm} \leq
\hspace{0.2cm}l_\gamma \parallel \nabla\phi_D\parallel$.
\end{proof}

\subsection{Isoperimetric Inequalities for groups of Polynomial growth}
\subsubsection{\textbf{A 2-dimensional Example}}

\paragraph{}
Here we will discuss the 2-dimensional example $\Z^2$. Let us consider the presentation $\langle a,b \hspace{0.1cm}|\hspace{0.1cm}[a,b]\rangle$ for $\Z^2$. Let $\widetilde{X}$ be the universal cover of the 2-complex $X$ corresponding to the presentation given above for $\Z^2$.
As before let us denote a ball of radius $r$ centered at the identity in $\widetilde{X}$ by $B_1(r)$.


Let us choose $r$ such that $ | B_1(r) |\geq 2 ||\phi_D||
 > | B_1(r-1) |$. Also,
$| B_1(r) | \hspace{0.2cm}\sim \hspace{0.2cm}\mathcal{O}(r^2)$.

From the propositions above, we already know that:\\
$\frac{1}{2}||\phi_D||\hspace{0.2cm} \leq T_D^{ \gamma_0}
\hspace{0.2cm} \leq \hspace{0.2cm}l_{\gamma_0} \hspace{0.2cm}
|\partial_{V} D |$ for some $\gamma_0 \in B_1(r)$.

$\hspace{2cm}\therefore ||\phi_D|| \hspace{0.2cm}\leq \hspace{0.2cm} 2
l_{\gamma_0 } \hspace{0.2cm} ||\nabla\phi_D||$

$\hspace{2cm}\therefore ||\phi_D|| \hspace{0.2cm}\preceq \hspace{0.2cm}||\phi_D||^{\frac{1}{2}}\hspace{0.2cm} |||\nabla\phi_D||$ ;
since $l_{\gamma_0} \leq \hspace{0.2cm} r$

$\hspace{2cm}\therefore ||\phi_D|| \hspace{0.2cm}\preceq
\hspace{0.2cm}\hspace{0.2cm}||\nabla\phi_D||^{2}$\\
When the 2-disc along with its boundary circle is
embedded in  $\widetilde{X}$ via the transverse map $f: (D^2,S^1) \rightarrow \widetilde{X}$, $||\nabla\phi_D||\hspace{0.2cm}= \hspace{0.2cm}
|\partial D^2|$. On the other hand,
in the case when we have a reduced diagram $f: (D^2,S^1) \rightarrow \widetilde{X}$, such that the disc and its boundary are not embedded, then by Lemma \ref{bound} $||\nabla\phi_D||\leq |\partial D^2|$. Hence we have the following isoperimetric inequality.

$\hspace{2cm}\therefore ||\phi_D|| \preceq
\hspace{0.05cm}
||\nabla\phi_D||^{2}\leq (2Const.)^{2}
 |\partial D^2|^{2}$.

$\hspace{2cm}\therefore Vol^2(D^2) \preceq |\partial D^2|^{2}$.

\subsubsection{\textbf{A 3-dimensional Example}}

\paragraph{}
In this section we present the an upper bound for the second-order Dehn functions of the 3-dimensional group $\mathcal{H}$ and consequently all cocompact lattices in the Nil geometry. In other words we complete the proof of Theorem \ref{th1} here.
Let $M$ be the 3-manifold corresponding to the lattice $\mathcal{H}$ in the Nil geometry mentioned above in Example \ref{lattice1} (along with the triangulation shown). Let $\widetilde{X}$ be its universal cover. So, one can find numerous copies of $M$ inside $\widetilde{X}$. Let $B_n$ represent a ball of radius $n$ in $\widetilde{X}$.\\

\begin{lemma}$||\phi_D|| \hspace{0.05cm}\preceq
\hspace{0.05cm}||\nabla\phi_D||^\frac{4}{3}$
\end{lemma}
\begin{proof}
  Let $\Gamma$ denote the dual Cayley graph embedded in $\widetilde{X}$ corresponding to the generating set $\mathcal{A}_0$ defined in Lemma \ref{generator} part $(i)$ where $\widetilde{X}$ is the universal cover of the 3-complex corresponding to $\mathcal{H}$.
Let us consider the reduced 3-dimensional diagram $f:(D^3,S^2)\rightarrow \widetilde{X}$ (defined in Section \ref{transverse1}). Let
$D$ be the finite set of vertices in $\Gamma$ dual to the 0-handles present in the diagram mentioned above.
Next, let us choose a ball of radius $r$ in the graph $\Gamma$ such that $ | B_1(r) |\geq 2 ||\phi_D||>| B_1(r-1) |$
 , where $r>2$ is real and $r$ is sufficiently large. Also,
$| B_1(r) | \hspace{0.2cm}\sim \hspace{0.2cm}\mathcal{O}(r^4)$, (\cite{Gr5},\cite{Bass}).
From Section \ref{comp}, we already know that,
$\frac{1}{2}||\phi_D||\hspace{0.2cm} \leq T_D^{ \gamma_0}
\hspace{0.2cm} \leq \hspace{0.2cm}l_{\gamma_0} \hspace{0.2cm}
|\partial_{V} D|$ for some $\gamma_0 \in B_1(r)$. Also as $l_{\gamma_0} \leq \hspace{0.2cm} r$ and
$ r-1\leq( 2||\phi_D||)^\frac{1}{4}\Rightarrow r\preceq( ||\phi_D||)^\frac{1}{4}$ and we have the following,

$\hspace{2cm}\therefore ||\phi_D|| \hspace{0.2cm}\leq \hspace{0.2cm} 2
l_{\gamma_0 } \hspace{0.2cm} ||\nabla\phi_D||$

$\hspace{2cm}\therefore||\phi_D|| \hspace{0.2cm}\preceq \hspace{0.2cm} 2
||\phi_D||^{\frac{1}{4}}\hspace{0.2cm} ||\nabla\phi_D||$

$\hspace{2cm}\therefore ||\phi_D|| \hspace{0.2cm}\preceq
\hspace{0.2cm}||\nabla\phi_D||^\frac{4}{3}$.
\end{proof}




 \begin{proof} (Proof of Theorem \ref{th1})
 Given a reduced diagram $f:(D^3,S^2)\rightarrow \widetilde{X}$, if the 3-ball and its boundary sphere are embedded in $\widetilde{X}$, then $||\nabla\phi_D||=|\partial D^3|$. If they are not embedded then by Lemma \ref{bound}, $||\nabla\phi_D||\leq |\partial D^3|$. Hence we have the following inequality.\\

$\hspace{0.8cm}\therefore Vol^3(D^3)\preceq
 |\partial D^3 |^\frac{4}{3}$, where $|\partial D^3 |$ is the volume of the boundary sphere.
 Therefore, by the definition of $\delta^{(2)}$, if $x$ is the maximum number of $3$-cells
 in the boundary sphere, then $\delta^{(2)}(x) \preceq x^\frac{4}{3}$.
\end{proof}

\subsection{Isoperimetric Inequalities for groups of Exponential growth}

\paragraph{}
In this section we present the upper bound for the second-order Dehn functions of $\mathcal{S}$ and consequently all cocompact lattices
in the Sol geometry. In other words, the proof of Theorem \ref{th2} will be completed here.\\
\begin{lemma}$||\phi_D|| \hspace{0.2cm}\preceq \hspace{0.2cm} \ln (||\nabla\phi_D||)\hspace{0.2cm} ||\nabla\phi_D||$.
\end{lemma}
\begin{proof}
We start with  a reduced 3-dimensional diagram $(D^3,S^2)$, corresponding to a finitely presented group $G$. In this sub-section, we have a 3-dimensional example with exponential growth namely, the solvable group $\mathcal{S}$.

Let us choose $r$ such that $ |B_1(r)| \geq 2 ||\phi_D||> |B_1(r-1)| $,
$|B_1(r)| \hspace{0.1cm}\sim \hspace{0.1cm}Ce^{\ln(k)r}$, $k,C$ are both positive constants, (\cite{Wolf},\cite{Mil2}). Therefore we have,\\

$\hspace{2.2cm} ||\phi_D|| \geq\hspace{0.1cm}C e^{\ln(k)r}$.

$\hspace{2cm}\therefore r \preceq \ln(||\phi_D||) $

Next, from Section \ref{comp}, we already know that,
$\frac{1}{2}||\phi_D||\hspace{0.2cm} \leq T_D^{ \gamma_0}
\hspace{0.2cm} \leq \hspace{0.2cm}l_{\gamma_0} \hspace{0.2cm}
||\nabla\phi_D||$ for some $\gamma_0 \in B_1(r)$.

$\hspace{2cm}\therefore ||\phi_D|| \hspace{0.2cm}\leq \hspace{0.2cm} 2
l_{\gamma_0 } \hspace{0.2cm} ||\nabla\phi_D||$

$\hspace{2cm}\therefore ||\phi_D|| \hspace{0.2cm}\preceq\hspace{0.2cm} \ln(||\phi_D||)\hspace{0.2cm} ||\nabla\phi_D||$ ;
since $l_{\gamma_0} \leq \hspace{0.2cm} r\hspace{2cm} (*)$

As in the case of $\mathcal{H}$, we can say the in the embedded case $||\nabla\phi_D||=| \partial
D^3 |$, while in the immersed case we have $|| \nabla \phi_D || \leq | \partial D^3 |$, using the Lemma \ref{bound} above. Hence we have the following isoperimetric inequality,

Taking natural logarithm, $\ln$, on either side of $(*)$ we get,

$\hspace{2.4cm} \ln( ||\phi_D||) \hspace{0.2cm}\preceq \hspace{0.2cm} \ln( \ln(||\phi_D||
)\hspace{0.2cm} ||\nabla\phi_D||)$,

$\hspace{2cm} \therefore \ln( ||\phi_D||) \hspace{0.2cm}\preceq\hspace{0.2cm} \ln( \ln(||\phi_D||)\hspace{0.2cm}+ \hspace{0.2cm}\ln (||\nabla\phi_D||)$,

Now from $(*)$,for large values of $||\phi_D||$,

 $\hspace{2.4cm}\ln(||\phi_D||)\leq  \frac{||\phi_D||}{\ln(||\phi_D||)}\preceq ||\nabla\phi_D||$,

$\hspace{2cm} \therefore \ln(||\phi_D||) \hspace{0.2cm}\preceq \hspace{0.2cm} \ln (||\nabla\phi_D||)$

Again from $(*)$,

$\hspace{2.45cm} ||\phi_D|| \hspace{0.2cm}\preceq \hspace{0.2cm}\ln (||\nabla\phi_D||)\hspace{0.2cm} ||\nabla\phi_D||$.
\end{proof}

\begin{proof}(Proof of Theorem \ref{th2})
From the lemma above we have,
$ Vol^3(D^3)\hspace{0.2cm}\preceq\hspace{0.2cm}\ln (|\partial D^3|)\hspace{0.2cm}|\partial D^3| $, where $|\partial D^3 |$ is the volume of the boundary sphere. Therefore, by the definition of $\delta^{(2)}$, if $x$ is the maximum number of $3$-cells
 in the boundary spheres then $\delta^{(2)}(x) \preceq x\ln(x)$.
\end{proof}

\bibliography{transport3}

\begin{thebibliography}{10}

\bibitem{Alex}
{\sc A.~Alexandrov}, {\em A theorem on triangles in a metric space and some of
  its applications}, Trudy Mat. Inst. Steklov, 38 (1951), pp.~5--23.

\bibitem{Al}
{\sc J.~Alonso, X.~Wang, and S.~Pride}, {\em Higher-dimensional isoperimetric
  (or {D}ehn) functions of groups}, Journal Group Theory,  (1999), pp.~81--112.

\bibitem{Bog}
{\sc J.~M. Alonso, W.~A. Bogley, R.~M. Burton, S.~J. Pride, and X.~Wang}, {\em
  Second order {D}ehn functions of groups}, Quart. J. Math. Oxford Ser. (2), 49
  (1998), pp.~1--30.

\bibitem{Ol1}
{\sc J.-C. Birget, A.~Y. Ol'shanski{\u\i}, E.~Rips, and M.~V. Sapir}, {\em
  Isoperimetric functions of groups and computational complexity of the word
  problem}, Ann. of Math. (2), 156 (2002), pp.~467--518.

\bibitem{Bra}
{\sc N.~Brady, M.~Bridson, M.~Forester, and K.~Shankar}, {\em Snowflake groups,
  {P}erron-{F}robenius eigenvalues, and isoperimetric spectra}, Geometry and
  Topology,  (2009).

\bibitem{For1}
{\sc N.~Brady and M.~Forester}, {\em Density of {I}soperimetric {S}pectra},
  Preprint,  (2008).

\bibitem{short}
{\sc N.~Brady, T.~Riley, and H.~Short}, {\em The geometry of the word problem
  for finitely generated groups}, Advanced Courses in Mathematics. CRM
  Barcelona, Birkh\"auser Verlag, Basel, 2007.
\newblock Papers from the Advanced Course held in Barcelona, July 5--15, 2005.

\bibitem{Bri1}
{\sc M.~Bridson}, {\em Polynomial {D}ehn functions and the length of
  asynchronously automatic structures}, Proc. London Mathematical Society, 85
  (2002), pp.~441--466.

\bibitem{Bri2}
{\sc M.~Bridson and A.~Haefliger}, {\em Metric Spaces of Non-Positive
  Curvature}, vol.~319, Springer, 1999.

\bibitem{Bri}
{\sc M.~Bridson, S.~Salamon, et~al.}, {\em Invitations to {G}eometry and
  {T}opology}, Oxford {S}cience {P}ublications, 2002.

\bibitem{Bu}
{\sc S.~Buoncristiano, C.~Rourke, and B.~Sanderson}, {\em A geometric approach
  to homology theory}, London Mathematical Society Lecture Note Series,
  (1976), pp.~iii+149.

\bibitem{Co}
{\sc T.~Coulhon and L.~Saloff-Coste}, {\em Isoperimetrie pour les groupes et
  les varietes}, Revista matem\'{a}tica iberoamericana, 9 (1993), pp.~293--314.

\bibitem{Dehn}
{\sc M.~Dehn}, {\em \"{U}ber unendliche diskontunuierliche {G}ruppen}, Math.
  Ann., 71 (1912), pp.~116--144.

\bibitem{Ep}
{\sc D.~Epstein, J.~Cannon, D.~Holt, S.~Levy, M.~S. Paterson, and W.~P.
  Thurston}, {\em Word processing in groups}, Jones and Bartlett Publishers,
  Boston, MA, 1992.

\bibitem{For}
{\sc M.~Forester}, {\em Lecture {N}otes on {T}opological {M}ethods in {G}roup
  {T}heory}.

\bibitem{Gr3}
{\sc M.~Gromov}, {\em Filling {R}iemannian manifolds}, J. Differential Geom.,
  18 (1983), pp.~1--147.

\bibitem{Gr2}
\leavevmode\vrule height 2pt depth -1.6pt width 23pt, {\em Asymptotic
  invariants of infinite groups}, vol.~182 of London Math. Soc. Lecture Note
  Ser., Cambridge Univ. Press, Cambridge, 1993.

\bibitem{G}
\leavevmode\vrule height 2pt depth -1.6pt width 23pt, {\em
  Carnot-{C}arath\'eodory spaces seen from within}, vol.~144 of Progr. Math.,
  Birkh\"auser, 1996.

\bibitem{Gr}
\leavevmode\vrule height 2pt depth -1.6pt width 23pt, {\em Metric Structures
  for Riemannian and Non-Riemannian Spaces}, vol.~152 of Progress in
  Mathematics, Birkhauser, 1999.

\bibitem{Bass}
{\sc H.Bass}, {\em The degree of polynomial growth of finitely generated
  nilpotent groups}, Proceedings London Mathematical Society, 25(4) (1972).

\bibitem{Cammond}
{\sc J.~P. McCammond}, {\em A general small cancellation theory}, Internat. J.
  Algebra Comput., 10 (2000), pp.~1--172.

\bibitem{Gr5}
{\sc M.Gromov}, {\em Groups of polynomial growth and expanding maps}, Inst.
  Hautes ´Etudes Sci. Publ. Math., 53 (1981).

\bibitem{Mil}
{\sc J.~Milnor}, {\em Lectures on the {$h$}-cobordism theorem}, Notes by L.
  Siebenmann and J. Sondow, Princeton University Press, Princeton, N.J., 1965.

\bibitem{Mil2}
{\sc J.~Milnor}, {\em Growth of finitely generated solvable groups}, J.
  Differential Geometry,  (1968).

\bibitem{Ol}
{\sc A.~Y. Ol'shanski{\u\i}}, {\em Geometry of defining relations in groups},
  vol.~70 of Mathematics and its Applications (Soviet Series), Kluwer Academic
  Publishers Group, Dordrecht, 1991.
\newblock Translated from the 1989 Russian original by Yu. A. Bakhturin.

\bibitem{Pla}
{\sc J.~A.~F. Plateau}, {\em {S}tatique {E}xperimentale et {T}h\'{e}orique des
  {L}iquides {S}oumis aux {S}eules {F}orces {M}oleculaires}, 1873.
\newblock Paris, Gauthier-Villars.

\bibitem{Rips1}
{\sc E.~Rips}, {\em Generalized small cancellation theory and applications.
  {I}. {T}he word problem}, Israel J. Math., 41 (1982), pp.~1--146.

\bibitem{Rk}
{\sc C.~P. Rourke and B.~J. Sanderson}, {\em Introduction to piecewise-linear
  topology}, Springer-Verlag, New York, 1972.
\newblock Ergebnisse der Mathematik und ihrer Grenzgebiete, Band 69.

\bibitem{sapir}
{\sc M.~V. Sapir, J.~Birget, and E.~Rips}, {\em Isoperimetric and isodiametric
  functions of groups}, Ann. of Math. (2), 156 (2002), pp.~345--466.

\bibitem{Sm}
{\sc S.~Smale}, {\em Generalized {P}oincar\'{e}'s conjecture on dimensions
  greater than four}, Annals of Mathematics, 74 (1961), pp.~391--406.

\bibitem{Th}
{\sc W.~P. Thurston}, {\em Three-dimensional geometry and topology. {V}ol. 1},
  vol.~35 of Princeton Mathematical Series, Princeton University Press,
  Princeton, NJ, 1997.
\newblock Edited by Silvio Levy.

\bibitem{V}
{\sc N.~T. Varopoulos}, {\em Random walks and {B}rowninan motion on manifolds},
  Symposia Mathematica,(Cortona, 1984), Academic Press, New York,  (1987),
  pp.~97--109.

\bibitem{wall}
{\sc C.~T.~C. Wall}, {\em Finiteness conditions for {${\rm CW}$}-complexes},
  Ann. of Math. (2), 81 (1965), pp.~56--69.

\bibitem{Wang2}
{\sc X.~Wang}, {\em Second order {D}ehn functions of split extensions of the
  form {${\Bbb Z}\sp 2\rtimes\sb \phi F$}}, Comm. Algebra, 30 (2002),
  pp.~4121--4137.

\bibitem{Wolf}
{\sc J.~A. Wolf}, {\em Growth of finitely generated solvable groups and
  curvature of riemannian manifolds}, J . Differential Geometry,  (1968).

\end{thebibliography}
\bibliographystyle{siam}
\end{document}